\documentclass[10pt]{amsart}
 \usepackage[margin=3cm]{geometry}
\usepackage{hyperref}
\usepackage{pb-diagram}
\usepackage{amsmath, amssymb}
\usepackage{tikz}
\usetikzlibrary{matrix,arrows,decorations.pathmorphing}
\usepackage{paralist}
\usepackage{color}

\newtheorem{thm}{Theorem}
\newtheorem{lem}[thm]{Lemma}
\newtheorem{prp}[thm]{Proposition}

\theoremstyle{definition}
\newtheorem{df}[thm]{Definition}
\newtheorem{exa}[thm]{Example}
\newtheorem{cor}[thm]{Corollary}

\theoremstyle{remark}
\newtheorem*{rem}{Remark}

\numberwithin{equation}{section}
\numberwithin{thm}{section}

\DeclareMathOperator{\im}{im}
\DeclareMathOperator{\id}{id}
\DeclareMathOperator{\supp}{supp}

\DeclareMathOperator{\map}{map}

\DeclareMathOperator{\hocolim}{hocolim}
\DeclareMathOperator*{\Lhocolim}{hocolim}
\DeclareMathOperator{\colim}{colim}
\DeclareMathOperator*{\Lcolim}{colim}

\DeclareMathOperator{\Ob}{Ob}

\DeclareMathOperator{\Ch}{Ch}
\DeclareMathOperator{\Seq}{Seq}

\DeclareMathOperator{\len}{len}
\DeclareMathOperator{\nat}{nat}
\DeclareMathOperator{\beg}{beg}
\DeclareMathOperator{\qend}{end}

\DeclareMathOperator{\dom}{dom}
\DeclareMathOperator{\const}{const}
\DeclareMathOperator{\Free}{Free}
\DeclareMathOperator{\Vertic}{Vert}
\DeclareMathOperator{\Lip}{Lip}
\DeclareMathOperator{\dist}{dist}
\DeclareMathOperator{\PF}{PF}

\def\dTop{\mathbf{dTop}}

\def\hTop{\mathbf{hTop}}
\def\Set{\mathbf{Set}}

\def\vN{\vec{N}}
\def\vP{\vec{P}}

\def\R{\mathbb{R}}
\def\Z{\mathbb{Z}}

\def\Set{\mathbf{Set}}

\def\Top{\mathbf{Top}}

\def\fC{\mathfrak{C}}

\def\fR{\mathfrak{R}}

\def\cC{\mathcal{C}}

\def\bO{\mathbf{0}}
\def\bI{\mathbf{1}}

\def\bx{\mathbf{x}}
\def\by{\mathbf{y}}
\def\ba{\mathbf{a}}
\def\bb{\mathbf{b}}

\def\bn{\mathbf{n}}
\def\bc{\mathbf{c}}

\def\bbf{\mathbf{f}}

\def\vI{\vec{I}}

\def\cupdot{\mathop{\dot\cup}}

\def\bTam{\mathrm{Tam}}
\def\Tam{\mathrm{Tam}}
\def\bip{^*_*}

\title{Spaces of directed paths on pre-cubical sets II}
\author{Krzysztof Ziemia\'nski}
\thanks{University of Warsaw, Faculty of Mathematics, Informatics and Mechanics, ul. Banacha 2, 02-097 Warszawa, Poland.\\ E-mail: ziemians@mimuw.edu.pl. ORCID: 0000--0001--7695--4028.}

\begin{document}

\begin{abstract}
	For a given pre-cubical set ($\square$--set) $K$ with two distinguished vertices $\bO$, $\bI$, we prove that the space $\vP(K)_\bO^\bI$ of d-paths on the geometric realization of $K$ with source $\bO$ and target $\bI$ is homotopy equivalent to its subspace $\vP^t(K)_\bO^\bI$ of tame d-paths. When $K$ is the underlying $\square$--set of a Higher Dimensional Automaton $A$, tame d-paths on $K$ represent step executions of $A$. Then, we define the cube chain category of  $K$ and prove that its nerve is weakly homotopy equivalent to $\vP(K)_\bO^\bI$.
\end{abstract}

\maketitle

\section{Introduction}

A directed space, or a d-space \cite{Grandis-DHT1}, is a topological space $X$ with a distinguished family of paths $\vP(X)$, called d-paths, that contains all constant paths and is closed with respect to concatenations and non-decreasing reparametrizations. Directed spaces serve as models in concurrency: points of a given directed space represent possible states of the concurrent program, while d-paths represent its possible partial executions. An important question is what is the homotopy type of the space $\vP(X)_x^y$ of d-paths beginning at the point $x$ and ending at the point $y$. If $x$ and $y$ are the initial and the final state of the program, respectively, this is represents the ``execution space" of the program modeled by $X$. Also, calculating of some invariants of d-spaces, eg. component categories \cite{R-Inv, Z-Stable} and natural homology \cite{DGG}, requires knowledge of the homotopy types of d-path spaces between two particular points.

 In this paper, we consider this problem for d-spaces that are geometric realizations of pre-cubical sets, called also $\square$--sets. $\square$--sets play an important role in concurrency: Higher Dimensional Automata introduced by Pratt \cite{Pratt} are $\square$--sets equipped with a labeling of edges, and then executions of a Higher Dimensional Automaton can be interpreted as d-paths on the geometric realization of the underlying $\square$--set. Van Glabbeek \cite{vGlabbeek} has shown that many other models for concurrency (eg. Petri nets) can be translated to Higher Dimensional Automata and, therefore, to $\square$--sets.

The problem of calculating of the homotopy types of d-path spaces between two vertices of a $\square$--set was studied in several papers, eg. \cite{R-Simplicial,R-Simplicial2,Z-Perm,Z-VF}. All these results work only for special classes of $\square$--sets, like Euclidean complexes or proper $\square$--sets, i.e., ones that their triangulations are simplicial complexes. In this paper, we consider the general case.

For an arbitrary $\square$--set $K$ with two distinguished vertices $\bO$, $\bI$, we prove that the space of d-paths $\vP(|K|)_\bO^\bI$ with source $\bO$ and target $\bI$ is homotopy equivalent to its subspace $\vP^t(|K|)_\bO^\bI$ of tame d-paths  (Theorem \ref{t:Tame}). A d-path is tame if it can be divided into segments each of which runs from the initial to the final vertex of some cube. In terms of Higher Dimensional Automata, tame d-paths correspond to step executions: at every step, some number of actions is performed, while no other action is active. 
Then we define the cube chain category of $K$, denoted  $\Ch(K)$, and prove that the geometric realization of the nerve of $\Ch(K)$ is homotopy equivalent to the space of tame d-paths on $K$ (Theorem \ref{t:ChN}). This provides a combinatorial model for the execution space of $K$ (Theorem \ref{t:Main}), which can be used for explicit calculations of its homotopy type.  All these constructions are functorial with respect to $K$, regarded as an object in the category of bi-pointed $\square$--sets. This theorems generalize the results of \cite{Z-Perm}.

\subsubsection*{Organization of the paper and relatonship with \cite{Z-Perm}}

The paper consists of two parts. The main goal of the first part (Sections 3--6) is to prove the tamification theorem (Thm. \ref{t:Tame}) and of the second part (Sections 7-11), to prove Theorems \ref{t:ChN} and \ref{t:Main}. The general outline of this paper resembles that of \cite{Z-Perm}  but we need to use more subtle arguments here.

Fix a bi-pointed $\square$--set $K$ with the initial vertex $\bO$ and the final vertex $\bI$. All d-paths $\alpha\in\vP(K)_\bO^\bI$ has integral $L^1$--lengths \cite{R-Trace}, and the spaces of d-paths having length $n$, denoted $\vP(K;n)_\bO^\bI$, will be handled separately for every $n\in\Z_{\geq 0}$.

In Section 3 we define tracks, which are sequences of cubes such that some upper face of preceding cube is a lower face of the succeeding one. Then we prove that every d-path $\alpha\in\vP(K)_\bO^\bI$ is a concatenation of d-paths lying in the consecutive cubes of some track. In Section 4, we define the set of actions that correspond to a given track $\fC$. In Section 5, we introduce progress functions of tracks and investigate the relationship between progress functions of a track $\fC$ and d-paths lying in $\fC$. Then in Section 6, we construct, in a functorial way, a self-map of the space $\vP(K;n)_\bO^\bI$ that is homotopic to the identity map and maps all natural d-paths (i.e., parametrized by length) into tame d-paths. The proof of the latter statement essentially uses progress functions. Since the space of natural d-paths is homotopy equivalent to the space of all d-paths, this implies Theorem \ref{t:Tame}.

This argument is essentially different from the one used in the proof of the tamification theorem in the previous paper \cite[Theorem 5.6]{Z-Perm}. That follows from a similar result for d-simplicial complexes \cite{Z-Cub}, which is proved by constructing some self-deformation of a given d-simplicial complex. This self-deformation is given by a direct but complicated and non-functorial formula and it is not clear how this can be interpreted in terms of concurrent processes. While a general outline of the argument is similar, the tamification via progress functions is more intuitive: a d-path $\alpha$ lying in a track $\fC$ is deformed to a tame d-path by ``speeding up" the executions of the actions of $\fC$. 

In the second part, we introduce cube chains: sequences of cubes in $K$ such that the final vertex of the preceding cube is the initial vertex of the succeeding one. These constitute the special case of tracks. Then, we formulate the main result of the second part of the paper stating that the nerve of the category $\Ch(K)$ of cube chains on $K$ is weakly homotopy equivalent to the space of natural tame paths on $K$ (Theorem \ref{t:ChN}). A natural d-path $\alpha\in\vN(K)_\bO^\bI$ is tame if and only it lies in some cube chain, i.e., admits a presentation of $\alpha$ as a concatenation of d-paths lying in the consecutive cubes of a given cube chain $\bc$. Such a presentation is called a natural tame presentation of $\alpha$; the difficulty that arises here is that $\alpha$ may have many natural tame presentation in $\bc$. This is an essential difference with the situation considered in \cite{Z-Perm}, where there is a good cover of the space of natural tame d-paths $\vN^t(K)_\bO^\bI$ that indexed by the poset of cube chains on $K$, and the analogue of Theorem \ref{t:ChN} can be prove using Nerve Lemma. 

Instead of a good cover, we need to investigate the functor $G:=\vN_{[0,n]}(\square^{{\vee}^{(-)}})_\bO^\bI:\Ch(K)\to \Top$, where $G(\bc)$ is the space of natural tame presentations in $\bc$, equipped with a map $F^K_n:\colim\;F\to \vN_{[0,n]}^t(K)_\bO^\bI$. In Section 8, we investigate properties of natural tame presentations. In Sections 9 and 10, we apply the results of Section 8 to prove that the middle and the right-hand map in the sequence  (\ref{e:MainSequence})
\[
	\lvert \Ch(K)\rvert \longleftarrow \hocolim\; G \longrightarrow \colim\;G \xrightarrow{F^K_n} \vN_{[0,n]}^t(K)_\bO^\bI
\]
are weak homotopy equivalences. This is easy for the left-hand map, so Theorem \ref{t:ChN} follows. Some topological proofs are left to Section 11.

\section{Preliminaries}

In this section we recall definitions and introduce notation that is used later on. See \cite{FGR} or \cite{FGHMR} for a survey which covers most of these topics.

\subsection{d-spaces}
Grandis \cite{Grandis-DHT1} defines a d-space as a pair $(X,\vP(X))$, where $X$ is a topological space and $\vP(X)$ is a family of paths on $X$ that contains all constant paths and is closed with respect to concatenations and non-decreasing reparametrizations. In this paper, it is more convenient to use a slightly different, though equivalent, definition. \emph{A d-space} is a topological space $X$ equipped with a d-structure. \emph{A d-structure} on $X$ is a collection of families of paths $\{\vP_{[a,b]}(X)\}_{a<b\in \R}$,
\[
	\vP_{[a,b]}(X)\subseteq P_{[a,b]}(X):=\map([a,b],X),
\]
called \emph{d-paths}, such that
\begin{itemize}
	\item{every constant path $\const^x_{[a,b]}:[a,b]\ni t\mapsto x\in X$ is a d-path,}
	\item{for every non-decreasing function $f:[a,b]\to [c,d]$ and every d-path $\alpha\in \vP_{[c,d]}(X)$, the path $\alpha \circ f$ is a d-path, i.e., $\alpha \circ f\in \vP_{[a,b]}(X)$,}
	\item{if $a<b<c$ and $\alpha\in \vP_{[a,b]}(X)$, $\beta\in \vP_{[b,c]}(X)$ are d-paths such that $\alpha(b)=\beta(b)$, then \emph{the concatenation} of $\alpha$ and $\beta$}:
	\[
	(\alpha*\beta)(t)
	=\begin{cases}
		\alpha(t) & \text{for $t\in [a,b]$,}\\
		\beta(t) & \text{for $t\in [b,c]$.}
	\end{cases}
\]
is a d-path, i.e., $\alpha*\beta\in\vP_{[a,c]}(X)$.
\end{itemize}
This definition is equivalent to the original Grandis' definition: if $(X,\{\vP_{[a,b]}(X)\}_{a<b})$ is a d-space, then $(X,\vP_{[0,1]}(X))$ is a d-space in the Grandis' sense. If $(X,\vP(X))$ is a d-space in Grandis' sense, then by letting
\[
	\vP_{[a,b]}(X)=\{[a,b]\ni t\mapsto \alpha((t-a)/(b-a))\in X\;|\; \alpha\in \vP(X)\}
\]
we obtain the d-space as defined above. We will occasionally write $\vP(X)$ for $\vP_{[0,1]}(X)$.

The sets $\vP_{[a,b]}(X)$ are topological spaces, with the compact-open topology inherited from the space of all paths $P_{[a,b]}(X)=\map([a,b],X)$.

Given two d-spaces $X,Y$, a continuous map $f:X\to Y$ is \emph{a d-map} if $f\circ \alpha\in \vP_{[a,b]}(Y)$ for every d-path $\alpha\in \vP_{[a,b]}(X)$ and for all $a<b\in \R$. The family of d-spaces with d-maps forms the category $\dTop$, which is complete and cocomplete.

For a d-space $X$ and $x,y\in X$, denote by $\vP_{[a,b]}(X)_x^y\subseteq \vP_{[a,b]}(X)$  the subspace of d-paths $\alpha$ such that $\alpha(a)=x$ and $\alpha(b)=y$.

\emph{A bi-pointed d-space} is a d-space $X$ with two distinguished points $\bO_X,\bI_X\in X$, the initial one $\bO_X$ and the final one $\bI_X$.  A bi-pointed d-map is a d-map that preserves the initial and the final points. The category of bi-pointed d-spaces and bi-pointed maps will be denoted by $\dTop\bip$.

\subsection{Directed intervals and cubes}
Let $s<t\in\R$. \emph{The directed interval} is the d-space $\overrightarrow{[s,t]}$ such that $\vP_{[a,b]}(\overrightarrow{[s,t]})$ is the space of non-decreasing continuous functions $[a,b]\to [s,t]$. The d-space $\vI=\overrightarrow{[0,1]}$ will be called \emph{the directed unit interval}.

\emph{The directed $n$-cube} $\vI^n$ is the categorical product of $n$ copies of the directed unit interval. A path on $\vI^n$ is a d-path of all its coordinates are d-paths in $\vI$, i.e., they are non-decreasing functions.
Points of $\vI^n$ will be denoted by bold letters, if possible, and their coordinates are distinguished by the upper indices, so that, for example $\bx=(x^1,x^2,\dots,x^n)\in\vI^n$. A similar convention will be used for d-paths: for $\beta\in\vP(\vI^n)$, $\beta^i\in \vP(\vI)$ denotes the $i$--th coordinate of $\beta$. We will write $|\bx|$ for $\sum_{i=1}^n x^i$.

Whenever $\vI$ or $\vI^n$ are considered as bi-pointed d-spaces, their initial and final points are $0,1\in \vI$ and
$
	\bO=(0,\dots,0), \bI=(1,\dots,1)\in\vI^n,
$
respectively.

\subsection{Quotient d-spaces}
Let $X,Y$ be topological spaces and let $p:X\to Y$ be a quotient map. Assume that $X$ is equipped with a d-structure $\{\vP_{[a,b]}(X)\}_{a<b\in\R}$. \emph{The quotient d-structure} on the space $Y$ is defined in the following way: a path $\alpha\in P_{[a,b]}(Y)$ is a d-path if and only if there exist numbers $a=t_0<\dots<t_n=b$ and d-paths $\beta_i\in \vP_{[t_{i-1},t_i]}(X)$ such that $\alpha(t)=p(\beta_i(t))$ for $t\in[t_{i-1},t_i]$. The quotient d-structure is the smallest d-structure on $Y$ such that $p$ is a d-map. The space $Y$ with the quotient d-structure will be called \emph{the quotient d-space} of $X$.

\subsection{$\square$--sets}
\emph{A pre-cubical set}, or \emph{a $\square$--set} $K$ is a sequence of disjoint sets $(K[n])_{n\geq 0}$ with a collection of \emph{face maps} $(d^\varepsilon_{i}:K[n]\to K[n-1])$ for $n>0$, $\varepsilon\in\{0,1\}$, $i\in \{1,\dots,n\}$ such that $d_i^\varepsilon d_j^\eta = d_{j-1}^\eta d_i^\varepsilon$ for all $\varepsilon,\eta\in\{0,1\}$ and $i<j$. Elements of the sets $K[n]$ will be called \emph{$n$--cubes} or just \emph{cubes} and $0$--cubes will be called \emph{vertices}. \emph{The dimension} of a cube $c$ as the integer $\dim(c)$ such that $c\in K[\dim(c)]$.
For $\square$--sets $K$, $L$, \emph{a $\square$--map} $f:K\to L$ is a sequence of maps $f[n]:K[n]\to L[n]$ that commute with the face maps. The category of $\square$--sets and $\square$--maps will be denoted by $\square\Set$.

An example of a $\square$--set is \emph{the standard $n$--cube} $\square^n$, such that $\square^n[k]$ is the set of functions $\{1,\dots,n\}\to\{0,1,*\}$ that take value $*$ for exactly $k$ arguments. The face map $d^\varepsilon_i$ converts the $i$--th occurrence of $*$ into $\varepsilon$. The only element of $\square^n[n]$ will be denoted by $u_n$. 
For any $\square$--set $K$, there is a 1--1 correspondence between the set of $n$--cubes $K[n]$ and the set of $\square$--maps $\square^n\to K$: for every $c\in K[n]$ there exists a unique map $f_c:\square^n\to K$ such that $f_c(u_n)=c$.

\emph{The $k$--th skeleton} of $\square$--set $K$ is a sub--$\square$--set $K_{(k)}\subseteq K$ given by 
\begin{equation}
	K_{(k)}[n] = \begin{cases}
		K[n] & \text{for $n\leq k$,}\\
		\emptyset & \text{for $n>k$.}
	\end{cases}
\end{equation}
\emph{The boundary of the standard $n$--cube} is $\partial\square^n:=\square^n_{(n-1)}$.

Given a $\square$--set $K$, a subset $A=\{a_1<\dots<a_k\}\subseteq \{1,\dots,n\}$ and $\varepsilon\in\{0,1\}$, define \emph{the iterated face map}
\begin{equation}\label{e:ExtraD}
	d^\varepsilon_A:=d^\varepsilon_{a_1}\circ d^\varepsilon_{a_{2}}\circ \dots \circ d^\varepsilon_{a_k}:K[n] \to K[n-k].
\end{equation}
Denote $d^\varepsilon:=d^\varepsilon_{\{1,\dots,n\}}:K[n]\to K[0]$. For a cube $c\in K[n]$, $d^0(c)$ and $d^1(c)$ will be called \emph{the initial} and \emph{the final vertex} of $c$, respectively.

\emph{A bi-pointed $\square$--set} is a triple $(K,\bO_K,\bI_k)$, where $K$ is a $\square$--set and $\bO_K,\bI_K\in K[0]$ are its vertices; we will write $\bO$ and $\bI$ for $\bO_K$ and $\bI_K$ whenever it does not lead to confusion. The category of bi-pointed $\square$--sets and base-points-preserving $\square$--maps will be denoted by $\square\Set\bip$.

\subsection{Geometric realization}
Let $K$ be a $\square$--set. \emph{The geometric realization} of $K$ is the quotient d-space
\begin{equation}
	|K|=\left(\coprod_{n\geq 0} K[n]\times \vI^n\right)/\sim.
\end{equation}
The relation $\sim$ is generated by $(c,\delta^\varepsilon_i(\bx))\sim (d^\varepsilon_i(c),\bx)$ for all $n\geq 1$, $i\in\{1,\dots,n\}$, $\varepsilon\in\{0,1\}$, $c\in K[n]$ and $\bx\in \vI^{n-1}$,  where
\[
	\delta^\varepsilon_i:\vI^{n-1}\ni (x^1,\dots,x^{n-1})\mapsto (x^1,\dots,x^{i-1},\varepsilon,x^i,\dots,x^{n-1})\in \vI^n
\]
is \emph{the coface map}.
For $c\in K[n]$ and $\bx=(x^1,\dots,x^{n})\in \vI^n$, $[c;\bx]\in|K|$ denotes the point represented by $(c,\bx)$. Every point $p\in|K|$ admits a unique \emph{canonical presentation} $p=[c_p;\bx_p]$ such that $x_p^j\neq 0,1$ for all $j$. The cube $c_p$ will be called \emph{the carrier} of $p$.
Each presentation of $p$ has the form
\[
	p=[c', \delta^{\varepsilon_1}_{i_1}(\dots(\delta^{\varepsilon_r}_{i_r}(\bx_p))\dots)],
\]
where $c'\in K[\dim(c_p)+r]$ is a cube such that $d^{\varepsilon_1}_{i_1}(\dots(d^{\varepsilon_r}_{i_r}(c'))\dots)=c_p$.

For $A=\{a_1<\dots<a_k\}\subseteq \{1,\dots,n\}$ and $\varepsilon\in\{0,1\}$, \emph{the iterated coface map} is defined by the formula
\begin{equation}\label{e:DeltaMaps}
	\delta^\varepsilon_A:= \delta^\varepsilon_{a_{k}}\circ \delta^\varepsilon_{a_{k-1}}\circ \dots\circ \delta^\varepsilon_{a_1}: \vI^{n-k}\to \vI^n
\end{equation}
Notice that if 
\[
	\bar{A}=\{\bar{a}_1<\dots<\bar{a}_{n-k}\}=\{1,\dots,n\}\setminus A,
\]
and $\bx=(x^1,\dots,x^{n-k})\in\vI^{n-k}$, then
\begin{equation}\label{e:DeltaCoords}
	\delta^\varepsilon_A(\bx)^i=\begin{cases}
		\varepsilon & \text{for $i\in A$,}\\
		x^{j} & \text{for $i=\bar{a}_j\in \bar{A}$.}
	\end{cases}
\end{equation}
Also, we have $[d^\varepsilon_A(c);\bx]=[c;\delta^\varepsilon_A(\bx)]$ for all $c$, $\bx$ and $A$.

Every $\square$--map $f:K\to L$ induces the d-map
\begin{equation}
	|f|:|K|\ni [c;\bx] \mapsto [f(c);\bx]\in  |L|.
\end{equation}
Thus, the geometric realization defines the functors $|-|:\square\Set \to \dTop$ and  $|-|:\square\Set\bip \to \dTop\bip$.

We will usually skip the vertical bars and write $\vP(K)$ for $\vP(|K|)$.

\subsection{Presentations of d-paths}
Let $K$ be a $\square$--set. \emph{A presentation} of a d-path $\alpha\in \vP_{[a,b]}(\lvert K\rvert)$ consists of
\begin{itemize}
	\item{a sequence $(c_i)_{i=1}^l$ of cubes of $K$, and}
	\item{a sequence $(\beta_i\in \vP_{[t_{i-1},t_i]}(\vI^{\dim(c_i)}))_{i=1}^l$ of d-paths,}
\end{itemize}
such that
\begin{itemize}
	\item{ $a=t_0\leq t_1\leq \dots\leq t_{l-1}\leq t_l=b$, and}
	\item{ $\alpha(t)=[c_i;\beta_i(t)]$ for every $i\in \{1,\dots,l\}$ and $t\in[t_{i-1},t_i]$.}
\end{itemize}
We write such a presentation as
\begin{equation}\label{e:dPathPresentation}
	\alpha=\overset{t_0}{\phantom{*}}[c_1;\beta_1]\overset{t_1}{*}[c_2;\beta_2]\overset{t_2}{*}\dots\overset{t_{l-1}}{*} [c_l;\beta_l]\overset{t_l}{\phantom{*}}.
\end{equation}

Immediately from the definition of a quotient d-space follows that every d-path in $|K|$ admits a presentation.


\subsection{Length}
Let $K$ be a $\square$--set. \emph{The $L^1$--length}, or just \emph{the length} of a d-path $\alpha\in\vP_{[a,b]}(K)$ is defined as
\begin{equation}
	\len(\alpha)=\sum_{i=1}^l |\beta_i(t_i)|-|\beta_i(t_{i-1})|=\sum_{i=1}^l \sum_{j=1}^{\dim(c_i)} \beta_i^j(t_i)-\beta_i^j(t_{i-1}),
\end{equation}
for some presentation (\ref{e:dPathPresentation}) of $\alpha$. The length was introduced by Raussen \cite[Section 2]{R-Trace}. This definition does not depend on the choice of a presentation, and defines, for every $a<b\in\R$, a continuous function $\len:\vP(K)\to \R_{\geq 0}$ \cite[Proposition 2.7]{R-Trace}.
If $x,y\in |K|$ and d-paths $\alpha,\alpha'\in \vP_{[a,b]}(K)_x^y$ are d-homotopic (i.e., they lie in the same path-connected component of $\vP_{[a,b]}(K)_x^y$) then $\len(\alpha)=\len(\alpha')$. Furthermore, if $K$ is a bi-pointed $\square$--set, then the length of every d-path $\alpha\in \vP_{[a,b]}(K)_\bO^\bI$ is an integer. As a consequence, there is a decomposition
\begin{equation}\label{e:LengthDecomposition}
	\vP_{[a,b]}(K)_\bO^\bI \cong \coprod_{n\geq 0} \vP_{[a,b]}(K;n)_\bO^\bI,
\end{equation}
where $\vP_{[a,b]}(K;n)_\bO^\bI$ stands for the space of d-paths having length $n$.

\subsection{Naturalization}\label{ss:Nat}
We say that a d-path $\alpha\in\vP_{[a,b]}(K)$ is \emph{natural} if $\len(\alpha|_{[c,d]})=d-c$ for every $a\leq c\leq d\leq b$. Let $\vN_{[a,b]}(K)\subseteq \vP_{[a,b]}(K)$ denote the subspace of natural d-paths.

Natural d-paths were introduced and studied by Raussen \cite{R-Trace}. He proved that for every d-path $\alpha\in \vP_{[a,b]}(K)$ there exists a unique natural d-path $\nat(\alpha)\in \vN_{[0,\len(\alpha)]}(K)$ such that
\begin{equation}
	\alpha(t)=\nat(\alpha)(\len(\alpha|_{[a,t]})).
\end{equation}
Furthermore, for a bi-pointed $\square$--set $K$, the map
\begin{equation}\label{e:natKn}
	\nat^K_n: \vP_{[0,n]}(K;n)_\bO^\bI\ni \alpha \mapsto \nat(\alpha)\in \vN_{[0,n]}(K)_\bO^\bI
\end{equation}
is a homotopy inverse of the inclusion map \cite[Propositions 2.15 and 2.16]{R-Trace}. The map $\nat^K_n$ is functorial with respect to $K\in\square\Set\bip$.

\subsection{Tame paths}
A d-path $\alpha\in\vP_{[a,b]}(K)_\bO^\bI$ is \emph{tame} if it admits a presentation (\ref{e:dPathPresentation}) such that, for every $i\in\{1,\dots,l-1\}$, $\alpha(t_i)$ is a vertex, i.e., has the form $[v;()]$ for some $v\in K[0]$. This definition generalizes the earlier definitions for d-paths on d-simplicial complexes \cite{Z-Cub} and on proper $\square$--sets \cite{Z-Perm}.

 If $\alpha$ is tame, then one can impose even a stronger condition on its presentation. For every $i$, all coordinates of $\beta_i(t_i)$ are either 0 or 1; if $\beta_i^j(t_i)=0$, then $\beta_i^j(t)=0$ for all $t\in [t_{i-1},t_i]$. Hence the segment $[c_i,\beta_i]$ may be replaced by $[d^0_j(c_i), \beta'_i]$, where $\beta'_i$ is the path obtained from $\beta_i$ by skipping its $j$--th coordinate. By repeating this operation, we obtain a presentation such that $\beta_i(t_i)=(1,\dots,1)$ for all $i$. In a similar way, we can guarantee that $\beta_i(t_{i-1})=(0,\dots,0)$.
A presentation (\ref{e:dPathPresentation}) such that $\beta_i(t_{i-1})=\bO$ and $\beta_i(t_i)=\bI$ for all $i$ will be called \emph{a tame presentation} of $\alpha$.

Let $\vP_{[a,b]}^t(K)_\bO^\bI$ (resp. $\vN^t_{[a,b]}(K)_\bO^\bI$) denote the space of all tame (resp. natural tame) d-paths on $K$ from $\bO$ to $\bI$.

\section{Tracks}

Let $K$ be a bi-pointed $\square$--set. Every d-path $\alpha\in\vP(K)_\bO^\bI$ is a concatenation of d-paths having the form $[c;\beta]$, where $c$ is a cube of $K$ and $\beta$ is a d-path in $\vI^{\dim(c)}$. We will show that we can impose extra conditions on the cubes and the paths which appear in such a presentation.

\begin{df}\label{d:Track}
	\emph{A track} $\fC$ in $K$ is a sequence of triples $(c_i, A_i, B_i)_{i=1}^l$, where
\begin{itemize}
	\item{$c_1,\dots,c_l$ are cubes of $K$,}
	\item{$A_i,B_i\subseteq \{1,\dots,\dim(c_i)\}$}
\end{itemize}
such that
\begin{enumerate}[\normalfont (a)]
	\item{$d^0_{A_1}(c_1)=\bO_K$,}
	\item{$d^1_{A_l}(c_l)=\bI_K$,}
	\item{$d^1_{B_i}(c_i)=d^0_{A_{i+1}}(c_{i+1})$ for every $i\in\{1,\dots,l-1\}$.}
	\item{the sets $A_1$, $B_l$ and $B_{i}\cup A_{i+1}$ are non-empty. \label{i:Nonempty}}
\end{enumerate}
\end{df}

\begin{prp}
	For every track $\fC=(c_i,A_i,B_i)_{i=1}^l$ we have $\sum_{i=1}^l |A_i|=\sum_{i=1}^l |B_i|$.
\end{prp}
\begin{proof}
	Conditions (a), (b) and (c) imply that $\dim(c_1)=|A_1|$, $\dim(c_{i+1})=\dim(c_i)-\lvert B_i\rvert +|A_{i+1}|$ and $\dim(c_l)=|B_{l}|$. We have
	\[
		|B_l|=\dim(c_l)=\dim(c_1)+\sum_{i=2}^{l} \lvert A_i\rvert - \sum_{i=1}^{l-1} |B_i|=\sum_{i=1}^{l}|A_i| - \sum_{i=1}^{l-1} |B_i|.\qedhere
	\]
\end{proof}
The integer $\sum |A_i|=\sum |B_i|$ will be called \emph{the length} of the track $\fC$ and denoted $\len(\fC)$.

\begin{df}\label{d:LiesInTrack}
Let $\fC=(c_i,A_i,B_i)_{i=1}^l$ be a track in $K$ and let $\alpha\in\vP_{[a,b]}(K)_\bO^\bI$. Denote $e_i=d^1_{B_i}(c_i)=d^0_{A_{i+1}}(c_{i+1})$. A presentation
\[
	\alpha=\overset{t_0}{\phantom{*}}[c_1;\beta_1]\overset{t_1}{*}[c_2;\beta_2]\overset{t_2}{*}\dots\overset{t_{l-1}}{*} [c_l;\beta_l]\overset{t_l}{\phantom{*}}
\]
is \emph{a $\fC$--presentation} of $\alpha$ is there exist points $\bx_i\in \vI^{\dim(e_i)}$, $i\in \{1,\dots,l-1\}$ such that for all $i$:
\begin{enumerate}
	\item{$\beta_i(t_i)=\delta^1_{B_i}(\bx_i)$,}
	\item{$\beta_{i+1}(t_i)=\delta^0_{A_{i+1}}(\bx_i)$.}
\end{enumerate}
We say that $\alpha$ \emph{lies in $\fC$} if it admits a $\fC$--presentation. The space of d-paths lying in $\fC$ will be denoted by $\vP_{[a,b]}(K,\fC)$.
\end{df}

\begin{prp}	
	If $\alpha\in\vP_{[a,b]}(K;\fC)$, then $\len(\alpha)=\len(\fC)$.
\end{prp}
\begin{proof}
	Choose a $\fC$--presentation of $\alpha$, as in Definition \ref{d:LiesInTrack}. For every $i\in\{1,\dots,l\}$ we have 
	\[
		\len(\alpha|_{[t_{i-1},t_i]})=\len(\beta_i)=|\beta_i(t_i)|-|\beta_i(t_{i-1})|=|B_i|+|\bx_i| - |\bx_{i-1}|,
	\]
	when assuming $|\bx_0|=|\bx_l|=0$. Thus,
	\[
		\len(\alpha)=\sum_{i=1}^l \len(\alpha|_{[t_{i-1},t_i]})=\sum_{i=1}^l |B_i|=\len(\fC).\qedhere
	\]
\end{proof}

Fajstrup proved \cite[2.20]{F-Dipaths} that if $K$ is a geometric $\square$--set, then every d-path $\alpha\in\vP(K)_\bO^\bI$ lies in a track called the carrier sequence of $\alpha$. Below we prove an analogue of these result for arbitrary $\square$--sets.

\begin{prp}
	Every non-constant d-path $\alpha\in\vP_{[a,b]}(K)_\bO^\bI$ lies in some track.
\end{prp}
\begin{proof}
	Choose a presentation
\[
	\alpha=\overset{t_0}{\phantom{*}}[c_1;\beta_1]\overset{t_1}{*}[c_2;\beta_2]\overset{t_2}{*}\dots\overset{t_{l-1}}{*} [c_l;\beta_l]\overset{t_l}{\phantom{*}}.	
\]	
	such that the integer $l+\sum_{i=1}^l \dim(c_i)$ is minimal among all presentations of $\alpha$. Then:
	\begin{enumerate}
		\item{
			$\beta_i^j(t_i)>0$ for all $i\in\{1,\dots,l\}$, $j\in\{1,\dots,\dim(c_i)\}$. Otherwise, there exist $i,j$ such that $\beta^j_i(t_i)=0$, which implies that $\beta_i^j(t)=0$ for all $t\in[t_{i-1},t_i]$. Hence, the segment $[c_i,\beta_i]$ may be replaced by
\[
	[d^0_j(c_i), (\beta_i^1,\dots,\beta_i^{j-1},\beta_i^{j+1},\dots,\beta_i^{\dim(c_i)})],
\]		
		 which contradicts the assumption that $l+\sum \dim(c_i)$ is minimal.
		 }
		 \item{
		 	$\beta_i^j(t_{i-1})<1$ for all $i\in\{1,\dots,l\}$, $j\in\{1,\dots,\dim(c_i)\}$; the argument is similar.
		 }
		 \item{
		 	For every $i\in\{1,\dots,l-1\}$, there exists $j\in\{1,\dots,\dim(c_i)\}$ such that $\beta_i^j(t_i)=1$, or there exists $k\in\{1,\dots,\dim(c_{i+1})\}$ such that $\beta_{i+1}^j(t_i)=0$. Assume otherwise; since $0<\beta_i^j(t_i), \beta_{i+1}^k(t_i)<1$ for all $j,k$, $[c_i,\beta_i(t_i)]$ and $[c_{i+1},\beta_{i+1}(t_i)]$ are both canonical presentations of the point $\alpha(t_i)$. Thus, $c_i=c_{i+1}$, $\beta_i(t_i)=\beta_{i+1}(t_{i})$ and the two segments $[c_i;\beta_i]$ and $[c_{i+1};\beta_{i+1}]$ may be replaced by the single segment $[c_i; \beta_i*\beta_{i+1}]$.
		 }
		 \item{$\dim(c_1),\dim(c_l)>0$. Otherwise the segment $[c_1;\beta_1]$ (resp. $[c_l,\beta_l]$) could be merged with $[c_2;\beta_2]$ (resp. $[c_{l-1},\beta_{l-1}]$), which exists since $\alpha$ is not constant.}
	\end{enumerate}
	Let
	\begin{align*}
		A_i&=\{j\in\{1,\dots,\dim(c_i)\}\;|\; \beta_i^j(t_{i-1})=0\}\\
		B_i&=\{j\in\{1,\dots,\dim(c_i)\}\;|\; \beta_i^j(t_{i})=1\}.
	\end{align*}
We will check that $\fC=(c_i,A_i,B_i)_{i=1}^l$ is a track. We have $\bO_K=[c_1;\beta_1(a)]$ and $\beta_i^j(a)<1$ for all $j$; thus, $\beta_i^j(a)=0$ (since $\alpha(t_0)=\bO_K$ is a vertex). Therefore, $A_1=\{1,\dots,\dim(c_1)\}$ and $d^0_{A_1}(c_1)=\bO_K$. Similarly, $B_l=\{1,\dots,\dim(c_l)\}$ and $d^1_{B_l}(c_l)=\bI_K$. Thus, conditions (a) and (b) of Definition \ref{d:Track} are satisfied.  Condition \ref{d:Track}.(d) follows from (3). From (1) and (2) follows that, for every $i$, there exist:
\begin{itemize}
	\item{a unique $\bx_i\in \vI^{\dim(c_i)-|B_i|}$ such that $\delta^1_{B_i}(\bx_i)=\beta_i(t_i)$, and}
	\item{a unique $\by_i\in \vI^{\dim(c_{i+1})-|A_{i+1}|}$ such that $\delta^0_{A_{i+1}}(\by_i)=\beta_{i+1}(t_i)$.}
\end{itemize}
We have
		\begin{align*}
			\alpha(t_i)&=[c_i;\beta_i(t_i)]=[c_i;\delta^1_{B_i}(\bx_i)]=[d^1_{B_i}(c_i);\bx_i]\\
			\alpha(t_i)&=[c_{i+1};\beta_{i+1}(t_i)]=[c_{i+1};\delta^0_{A_{i+1}}(\by_i)]=[d^0_{A_{i+1}}(c_{i+1});\by_i].
		\end{align*}
	All coordinates of both $\bx_i$ and $\by_i$ are different than either 0 or 1. Thus, $[d^1_{B_i}(c_i);\bx_i]$ and $[d^0_{A_{i+1}}(c_{i+1});\by_i]$ are both canonical presentations of the same point and, therefore, they are equal. As a consequence, $d^1_{B_i}(c_i)=d^0_{A_{i+1}}(c_{i+1})$, which proves that \ref{d:Track}.(c) is satisfied. Thus, $\fC$ is a track. Moreover, the points $\bx_i=\by_i$ fit into the Definition \ref{d:LiesInTrack}; hence, $\alpha$ lies in $\fC$.
\end{proof}

\section{Actions}

Every d-path $\alpha$ between vertices of a $\square$--set $K$ having length $n$ can be interpreted as a performance of $n$ different actions. This is an easy observation if $K$ is a Euclidean complex in the sense of \cite{RZ}. In this section we will show how to interpret  actions when $K$ is an arbitrary $\square$--set.

Fix a bi-pointed $\square$--set $K$ and a track $\fC=(c_i,A_i,B_i)_{i=1}^l$ in $K$ having length $n$. For $i\in\{1,\dots,l\}$, denote
\begin{equation*}
	 q_i=\dim(c_i)-|B_i|=\dim(c_{i+1})-|A_{i+1}|,
\end{equation*}
\begin{align}\label{e:BarAB}
	\bar{A}_i&=\{\bar{a}_i^1<\bar{a}_i^2<\dots<\bar{a}_i^{q_{i-1}}\}=\{1,\dots,\dim(c_i)\}\setminus A_i,\\
	\bar{B}_i&=\{\bar{b}_i^1<\bar{b}_i^2<\dots<\bar{b}_i^{q_i}\}=\{1,\dots,\dim(c_i)\}\setminus B_i.\notag
\end{align}

Consider the set of pairs $(i,r)$ such that $i\in \{1,\dots,k\}$, $r\in \{1,\dots,\dim(c_i)\}$. We will call these pairs \emph{local $\fC$--actions} (or \emph{local actions} if $\fC$ is clear).
Let $\sim$ be the equivalence relation on the set of local $\fC$--actions generated by
\begin{equation}\label{e:ActionRelation}
	(i,\bar{b}_i^{j})\sim (i+1,\bar{a}_{i+1}^j)
\end{equation}
for all $i\in\{1,\dots,l-1\}$, $j\in\{1,\dots,q_i\}$.

\begin{df}\label{d:Action}
	\emph{A $\fC$--action} (or an action if $\fC$ is clear) is an equivalence class of the relation $\sim$. The set of all $\fC$--actions will be denoted by $T(\fC)$. The $\fC$--action represented by $(i,r)$ will be denoted $[i,r]$.
\end{df}

The following proposition justifies the definition above.
\begin{prp}
	Fix $i\in\{1\dots,l-1\}$. Let $\bx=(x^1,\dots,x^{\dim(c_i)})\in I^{\dim(c_{i})}$, $\by=(y^1,\dots,y^{\dim(c_{i+1})})\in I^{\dim(c_{i+1})}$ be points such that
	\begin{itemize}
	\item{$x^j=1$ for every $j\in B_{i}$,}
	\item{$y^k=0$ for every $k\in A_{i+1}$,}
	\item{$x^j=y^k$ whenever $[i,j]=[i+1,k]$.}
	\end{itemize}
	Then $[c_i;\bx]=[c_{i+1};\by]$.
\end{prp}
\begin{proof}
	Since $x^j=1$ for all $j\in B_i$ and $y^k=0$ for all $k\in A_{i+1}$, we have
\[
	\bx=\delta^1_{B_i}(x^{\bar{b}_i^1},x^{\bar{b}_i^2},\dots,x^{\bar{b}_i^{q_i}}), \qquad \by=\delta^0_{A_{i+1}}(y^{\bar{a}_{i+1}^1}, y^{\bar{a}_{i+1}^2},\dots,y^{\bar{a}_{i+1}^{q_i}})
\]	
For all $r\in\{1,\dots,q_i\}$, we have  $[i,\bar{b}_i^r]=[i+1,\bar{a}_{i+1}^r]$. Therefore,
\[
	[c_i;\bx]=[d^1_{B_i}(c_i);(x^{\bar{b}_i^1},x^{\bar{b}_i^2},\dots,x^{\bar{b}_i^{q_i}})]=[d^0_{A_{i+1}}(c_{i+1}); (y^{\bar{a}_{i+1}^1}, y^{\bar{a}_{i+1}^2},\dots,y^{\bar{b}_{i+1}^{q_i}})]=[c_{i+1},\by]. \qedhere
\]	
\end{proof}

Let us collect some basic properties of $\fC$--actions:

\begin{enumerate}
\item{	\label{i:FaceEquivalence} 
	For every $i\in\{1,\dots,l-1\}$, the sequences of actions
	\begin{equation}\label{e:FaceEquity}	
		([i,r])_{r\in \{1,\dots,\dim(c_i)\}\setminus B_i} \quad \text{and} \quad ([i+1,s])_{s\in \{1,\dots,\dim(c_{i+1})\}\setminus A_{i+1}}
	\end{equation}
	are equal.
}
\item{
	Every action $p\in T(\fC)$ has at most one representative having the form $(i,r)$ for a fixed $i$. If such a representative exists, its second coordinate will be denoted by $r(p,i)$, so that $p=[i,r(p,i)]$. In such a case we will say that the action $p$ is \emph{active at the $i$--th stage}.
}
\item{
	For a given action $p$, the set of stages at which $p$ is active forms a (non-empty) interval, i.e., has the form
	\begin{equation}
		\{i:\;\beg(p)\leq i \leq \qend(p)\}
	\end{equation}
	for some integers $1\leq \beg(p)\leq \qend(p)\leq l$.
	Denote
\begin{align}\label{e:T}
	T^1_i(\fC)&=\{p\in T(\fC)\;|\; \qend(p)< i\} \notag\\
	T^*_i(\fC)&=\{p\in T(\fC)\;|\; \beg(p)\leq i\leq \qend(p)\}\\
	T^0_i(\fC)&=\{p\in T(\fC)\;|\; i< \beg(p)\}.\notag
\end{align}
	These are the sets of actions that are \emph{finished}, \emph{active} and \emph{unstarted}, respectively, at the $i$--th stage.
}
\item{
	For every $i$, the following pairs of conditions are equivalent:
	\begin{align}\label{e:BegEndAB}
		 \beg(p)=i \;\Leftrightarrow \; \text{$p=[i,r]$ for some $r\in A_i$}\\
		 \qend(p)=i \;\Leftrightarrow \; \text{$p=[i,r]$ for some $r\in B_i$}.\notag
	\end{align}
}
\item{
	We have
	\begin{align}\label{e:ABSets}
		\{1,\dots,l\} \supsetneq T^0_1(\fC) & \supseteq T^0_2(\fC)\supseteq \dots\supseteq T^0_k(\fC)= \emptyset \\
		\emptyset = T^1_1(\fC) & \subseteq T^1_2(\fC)\subseteq \dots\subseteq T^1_k(\fC)\subsetneq\{1,\dots,l\}.\notag
	\end{align}
}
\end{enumerate}

\section{Progress functions}
In this section we introduce progress functions, which provide a convenient description of d-paths lying in a given track $\fC$. Fix a bi-pointed $\square$--set $K$, a track $\fC=(c_i,A_i,B_i)_{i=1}^l$ in $K$ and numbers $a<b\in\R$.

\begin{df}\label{d:ProgressFunction}
	\emph{A progress function} of $\fC$ is a sequence $\bbf=(f^p)_{p\in T(\fC)}$ of non-decreasing continuous functions $[a,b]\to [0,1]$ such that there exist numbers
	\[
		a=t_0\leq t_1\leq \dots \leq t_k=b
	\]
	such that for every $p\in T(\fC)$
	\begin{itemize}
		\item{$f^p(t)=0$ for $t\leq t_{\beg(p)-1}$,}
		\item{$f^p(t)=1$ for $t\geq t_{\qend(p)}$.}
	\end{itemize}
	Let $\PF_{[a,b]}(\fC)$ be the space of progress functions of the track $\fC$, with the compact-open topology.
\end{df}

For every progress function $\bbf\in \PF_{[a,b]}(\fC)$ and an action $p\in T(\fC)$ we have $f^p(a)=0$ and $f^p(b)=1$. Thus, $f^p$ can be regarded as a d-path in $\vP_{[a,b]}(\vI)_0^1$, and $\PF_{[a,b]}(\fC)$,  as a subspace of $\vP_{[a,b]}(\vI^{T(\fC)})_{\bO}^\bI$.
\emph{The support of $f^p$}, defined by
\[
	\supp(f^p):=\{t\in [a,b]\;|\; 0<f^p(t)<1\}
\]
 is an open interval. We will denote its endpoints by $a_\bbf^p$ and $b_\bbf^p$ so that
\begin{itemize}
\item{  $\supp(f^p)=(a_\bbf^p,b_\bbf^p)$,}
\item{$f^p(t)=0$ for $t\in [a,a_\bbf^p]$,}
\item{$f^p(t)=1$ for $t\in [b_\bbf^p,b]$.}
\end{itemize} 

For any sequence of numbers $(t_i)$ satisfying Definition \ref{d:ProgressFunction} we have
\begin{equation}\label{e:BegEnd}
	t_{\beg(p)-1}\leq a_\bbf^p<b_\bbf^p\leq t_{\qend(p)}.
\end{equation}
\emph{The support of a progress function} $\bbf$ is the set 
\begin{equation}
	\supp(\bbf)=\bigcup_{p\in T(C)}\supp(f^p).
\end{equation}

In the remaining part of this section we will describe the relationship between progress functions of $\fC$ and d-paths lying in $\fC$.
Let $\bbf=(f^p)_{p\in T(\fC)}$ be a progress function of $\fC$. Our goal is to construct a d-path $\alpha^\bbf\in \vP(K;\fC)_{\bO}^\bI$ that corresponds to $\bbf$. For every $i\in\{1,\dots,l\}$ let
\begin{equation}
	a^\bbf_i:=\max_{p\in T^1_i(\fC)} b_\bbf^p	, \qquad b^\bbf_i:=\min_{p\in T^0_i(\fC)} a_\bbf^p.
\end{equation}
If $T^1_i(\fC)=\emptyset$ (resp. $T^0_i(\fC)=\emptyset$), we take $a^\bbf_i=a$ (resp. $b^\bbf_i=b$). The interval $[a^\bbf_i,b^\bbf_i]$ is the maximal one that is disjoint with the supports of functions $f^p$ for actions $p$ that are not active at the $i$--th stage.
For every $i$ define a d-path $\beta^\bbf_i\in\vP_{[a^\bbf_i,b^\bbf_i]}(\vI^{\dim(c_i)})$ by
\begin{equation}\label{e:AlphaFDef}
	\beta^{\bbf}_i(t)=(f^{[i,1]}(t), f^{[i,2]}(t), \dots, f^{[i,\dim(c_i)]}(t) ),
\end{equation}
and let $\alpha^\bbf_i=[c_i;\beta^\bbf_i]\in\vP_{[a^\bbf_i,b^\bbf_i]}(K)$ .

\begin{prp}\label{p:ValidityOfProgressPieces}
	For every progress function $\bbf\in F_{[a,b]}(\fC)$:
	\begin{enumerate}[\normalfont (a)]
		\item{
			$[a,b]=\bigcup_{i=1}^l [a^\bbf_i,b^\bbf_i]$,
		}
		\item{
			$\alpha^\bbf_i(t)=\alpha^\bbf_j(t)$ for all $i$, $j$ and $t\in [a^\bbf_i,b^\bbf_i]\cap [a^\bbf_j,b^\bbf_j]$.
		}
	\end{enumerate}
\end{prp}
\begin{proof}
	By (\ref{e:ABSets}) we have $a=a^\bbf_1\leq a^\bbf_2\leq \dots \leq a^\bbf_l$ and $b^\bbf_1\leq b^\bbf_2\leq \dots\leq b^\bbf_l=b$. Choose a sequence $(t_i)_{i=0}^l$ satisfying the condition in Definition \ref{d:ProgressFunction}. For every $i\in\{1,\dots,l-1\}$ we have
	\[
		a^\bbf_{i+1}=\max_{p\in T^1_{i+1}(\fC)}b^\bbf_p \overset{(\ref{e:BegEnd})}\leq \max_{p\in T^1_{i+1}(\fC)} t_{\qend(p)} \overset{(\ref{e:T})}\leq t_i
		\overset{(\ref{e:T})}\leq \min_{p\in T^0_i(\fC)} t_{\beg(p)-1} \overset{(\ref{e:BegEnd})}\leq \min_{p\in T^0_i(\fC)} a^\bbf_p = b^\bbf_i.
	\]
	This implies (a). To prove (b), it is enough to check that $\alpha^\bbf_i(t)=\alpha^\bbf_{i+1}(t)$ for every $t\in [a^\bbf_{i+1},b^\bbf_i]$.
	For $r\in\{1,\dots,\dim(c_i)\}$ we have 
	\[
		r\in B_i \;\overset{(\ref{e:BegEndAB})}\Rightarrow\;
		\qend([i,r])=i \;\Rightarrow\;
		[i,r]\in T^1_{i+1}(\fC) \;\Rightarrow\;
		a^\bbf_{i+1}\geq b^{[i,r]}_\bbf \;\Rightarrow\;
		f^{[i,r]}(a^\bbf_{i+1})=1.
	\]
	Thus, $f^{[i,r]}(t)=1$ for all $r\in B_i$. As a consequence, 
	\[
		(f^{[i,1]}(t),\dots,f^{[i,\dim(c_i)]})=\delta^1_{B_i}(f^{[i,\bar{b}_i^1]}(t),\dots,f^{[i,\bar{b}_i^{q_i}]}(t)),
	\]
	where $\bar{b}_i^j$ and $q_i$ are defined in (\ref{e:BarAB}). Finally,
	\begin{multline*}
		\alpha^\bbf_i(t)=[c_i;(f^{[i,1]}(t),\dots,f^{[i,\dim(c_i)]})]=[c_i;\delta^1_{B_i}(f^{[i,\bar{b}_i^1]}(t),\dots,f^{[i,\bar{b}_i^{q_i}]}(t))]\\
		= [d^1_{B_i}(c_i);(f^{[i,\bar{b}_i^1]}(t),\dots,f^{[i,\bar{b}_i^{q_i}]}(t))].
	\end{multline*}
	In a similar way we can show that
	\begin{equation*}
		\alpha^\bbf_{i+1}(t)=[d^0_{A_{i+1}}(c_{i+1});(f^{[i+1,\bar{a}_{i+1}^1]}(t), \dots, f^{[i+1,\bar{a}_{i+1}^{q_i}]}(t))].
	\end{equation*}
	Since $d^0_{A_{i+1}}(c_{i+1})=d^1_{B_i}(c_i)$ and $[i,\bar{b}^j_i]=[i+1,\bar{a}^j_{i+1}]$ for all $j$, the conclusion follows.
\end{proof}

\begin{df}
	\emph{The d-path associated to a progress function} $\bbf\in \PF_{[a,b]}$ is the unique d-path $\alpha^\bbf\in\vP_{[a,b]}(K;\fC)$ such that $\alpha(t)=\alpha^\bbf_i(t)$ for every $t\in [a^\bbf_i,b^\bbf_i]$. Proposition \ref{p:ValidityOfProgressPieces} guarantees that $\alpha^\bbf$ exists and is determined uniquely.
\end{df}

The construction presented above defines the map
\begin{equation}\label{e:ProgPathMap}
	\fR_\fC:\PF_{[a,b]}(\fC) \ni \bbf \mapsto \alpha^\bbf \in \vP_{[a,b]}(K;\fC),
\end{equation}		
which can be shown to be continuous. We skip a proof of this fact since it is tedious and not necessary for proving the main results of this paper.

Now we will construct a progress function associated to a given d-path $\alpha\in\vP_{[a,b]}(K;\fC)$. Choose a $\fC$--presentation
\begin{equation}\label{e:PathPresentation}
		\alpha=\overset{t_0}{\phantom{*}}[c_1;\beta_1]\overset{t_1}{*}[c_2;\beta_2]\overset{t_2}{*}\dots\overset{t_{l-1}}{*} [c_l;\beta_l]\overset{t_l}{\phantom{*}}.
\end{equation}

\begin{prp}\label{p:BetaGlues}
	Let $p\in T(\fC)$. Then
	\begin{enumerate}[\normalfont (a)]
	\item{$\beta_i^{r(p,i)}(t_i)=\beta_{i+1}^{r(p,i+1)}(t_i)$ whenever $\beg(p)\leq i <\qend(p)$,}
	\item{$\beta_i^{r(p,i)}(t_{i-1})=0$ for $i=\beg(p)$,}
	\item{$\beta_i^{r(p,i)}(t_i)=1$ for $i=\qend(p)$.}
	\end{enumerate}	
\end{prp}
\begin{proof}
	Assume that $\beg(p)\leq i <\qend(p)$. Then $r(p,i)\not\in B_i$ and there exists $j\in\{1,\dots,q_i\}$ such that $r(p,i)=\bar{b}_i^j$. Furthermore, $r(p,i+1)=\bar{a}_{i+1}^j$. Let $\bx_i\in \vI^{q_i}$ be a point such that $\beta_i(t_i)=\delta^1_{B_i}(\bx_i)$ and $\beta_{i+1}(t_i)=\delta^0_{A_{i+1}}(\bx_i)$. We have
	\[
		\beta^{r(p,i)}_i(t_i)=
		\beta_i^{\bar{b}^j_i}(t_i)=
		\delta^1_{B_i}(\bx_i)^{\bar{b}^j_i} \overset{(\ref{e:DeltaCoords})}=
		x_i^j\overset{(\ref{e:DeltaCoords})}=
		\delta^0_{A_{i+1}}(\bx_i)^{\bar{a}^j_{i+1}}=
		\beta_{i+1}^{\bar{a}^j_{i+1}}(t_i)=
		\beta^{r(p,i+1)}_{i+1}(t_i).
	\]
	If $i=\beg(p)$, then $r(p,i)\in A_i$ and then 
	\[
		\beta^{r(p,i)}_i(t_{i-1})=\delta^0_{A_i}(\bx_i)^{r(p,i)}=0.
	\]
	A similar argument shows (c).
\end{proof}

 For an arbitrary action $p\in T(\fC)$ let us define the function
\begin{equation}
	f^p_\alpha:[a,b]\ni t \mapsto
	\begin{cases}
		0 & \text{for $t\leq t_{\beg(p)-1}$}\\
		\beta_i^{r(p,i)}(t) & \text{for $\beg(p)\leq i \leq \qend(p)$ and  $t\in [t_{i-1},t_i]$}\\
		1 & \text{for $t\geq t_{\qend(p)}$}\\
	\end{cases}
\end{equation}
and let $\bbf_\alpha=(f_\alpha^p)_{p\in T(\fC)}$. Proposition \ref{p:BetaGlues} implies that this definition is valid.

\begin{prp}\label{p:ProgressOfDPath}
	$\bbf_\alpha$ is a progress function on $\fC$. Moreover, $\alpha^{\bbf_\alpha}=\alpha$.
\end{prp}
\begin{proof}
	The sequence $(t_i)$ from the presentation (\ref{e:PathPresentation}) satisfies the conditions required in Definition \ref{d:ProgressFunction}. Notice that  $[t_{i-1},t_i]\subseteq [a^{\bbf_\alpha}_i, b^{\bbf_\alpha}_i]$, for all $i\in\{1,\dots,l\}$. Then, for every $t\in [t_{i-1},t_i]$, $\alpha^{\bbf_\alpha}_i(t)$ is well-defined, and
	\begin{multline*}
		\alpha^{\bbf_\alpha}(t)=
		\alpha_i^{\bbf_\alpha}(t)=
		[c_i;(f_\alpha^{[i,1]}(t),\dots,f_\alpha^{[i,\dim(c_i)]}(t))]=
		[c_i;(\beta_i^{r([i,1],i)}(t),\dots,\beta_i^{r([i,\dim(c_i)],i)}(t))]=\\
		[c_i;(\beta_i^1(t),\dots,\beta_i^{\dim(c_i)}(t))]=
		[c_i;\beta_i(t)]=
		\alpha(t).\qedhere
	\end{multline*}
\end{proof}

The function $\bbf_\alpha$ will be called \emph{the progress function of $\alpha\in\vP(K;\fC)$} with the presentation (\ref{e:PathPresentation}). Proposition \ref{p:ProgressOfDPath} implies that the map $\fR_\fC$ is surjective. This map is not, in general, a bijection: $\bbf^\alpha$ depends not only on the d-path $\alpha$ but also on the choice of its presentation, as shown in the example below.

\begin{exa}
	Let $K$ be the $\square$--set having exactly one cube in dimensions 0, 1, 2 and no cubes in higher dimensions. Let $e$ be the only 2--dimensional cube and let $\fC=(c_1=e, A_1=\{1,2\}, B_1=\{1,2\})$. Let $\alpha\in\vP_{[0,2]}(K)$ be the d-path given by the $\fC$--presentation $\alpha=[e;\beta]$, where
	\[
		\beta(t)=\begin{cases}
			(t,0) & \text{for $t\in [0,1]$}\\
			(1,t-1) & \text{for $t\in [1,2]$.}
		\end{cases}
	\]
	 Then $T(\fC)=\{[1,1],[1,2]\}$ and the progress function of $\alpha$ is given by
	\[
		f_\alpha^{[1,1]}=\begin{cases}
			t & \text{for $t\in [0,1]$}\\
			1 & \text{for $t\in [1,2]$.}
		\end{cases},
		\qquad
		f_\alpha^{[1,2]}=\begin{cases}
			0 & \text{for $t\in [0,1]$}\\
			t-1 & \text{for $t\in [1,2]$.}
		\end{cases},
	\]
	But $\alpha$ has another $\fC$--presentation, namely $\alpha=[e;\beta']$, where
	\[
		\beta'(t)=\begin{cases}
			(0,t) & \text{for $t\in [0,1]$}\\
			(t-1,1) & \text{for $t\in [1,2]$,}
		\end{cases}
	\]
	which gives the progress function with $f_\alpha^{[1,1]}$ and $f_\alpha^{[1,2]}$ swapped.
\end{exa}

\section{Tamification theorem}

In this section we will use the results obtained above to prove that the spaces of d-paths and of tame d-paths are homotopy equivalent. 
The main result is the following:

\begin{thm}\label{t:Tame}
	For every $n\geq 0$ and every $K\in\square\Set\bip$, all the inclusions in the diagram
	\begin{equation}\label{e:TamingDiagram}
		\begin{diagram}
			\node{\vP^t_{[0,n]}(K;n)_\bO^\bI}
				\arrow{e,t}{\subseteq}
			\node{\vP_{[0,n]}(K;n)_\bO^\bI}
		\\
			\node{\vN^t_{[0,n]}(K)_\bO^\bI}
				\arrow{e,t}{\subseteq}
				\arrow{n,l}{\subseteq}
			\node{\vN_{[0,n]}(K)_\bO^\bI}		
				\arrow{n,r}{\subseteq}
			\end{diagram}
		\end{equation}
	are homotopy equivalences.
\end{thm}

For vertical maps this follows from \cite{R-Trace}.
The main idea of the proof of Theorem \ref{t:Tame} is to construct a functorial self-map of $\vP_{[0,n]}(K;n)_\bO^\bI$ that is homotopic to the identity and maps $\vN_{[0,n]}(K)_\bO^\bI$ into $\vP^t_{[0,n]}(K;n)_\bO^\bI$.

Let $R:\overrightarrow{[0,n]}\times \vI \to \vI$ be an arbitrary d-map such that $R(t,0)=0$ and $R(t,1)=1$ for all $t\in[0,n]$. For every $k\geq 0$, $R$ induces the map
\[
    R^k:\overrightarrow{[0,n]}\times \vI^k \ni (t;h^1,\dots,h^k) \mapsto (R(t,h^1),\dots,R(t,h^k))\in \vI^k.
\]
The maps $R^k$ are compatible with the face maps, i.e., for every $\varepsilon\in\{0,1\}$, $k\geq 0$ and $i\in\{1,\dots,k+1\}$, the diagram
\begin{equation}
	\begin{diagram}
		\node{\overrightarrow{[0,n]} \times \vI^k}
			\arrow{e,t}{R^k}
			\arrow{s,l}{\id\times \delta^\varepsilon_i}
		\node{\vI^k}
			\arrow{s,r}{\delta^\varepsilon_i}
	\\
		\node{\overrightarrow{[0,n]} \times \vI^{k+1}}
			\arrow{e,t}{R^{k+1}}
		\node{\vI^{k+1}}
	\end{diagram}
\end{equation}
commutes. As a consequence, for every $K\in\square\Set$, the maps $R^k$ induce the continuous d-map
\begin{equation}
	R^{K}: \overrightarrow{[0,n]} \times |K|\to |K|,
\end{equation}
such that
\begin{equation}
	R^K(t,[c;\bx])=[c;R^k(t,\bx)]=[c;R(t,x_1),\dots,R(t,x_k)]
\end{equation}
 for $c\in K[k]$, $\bx=(x^1,\dots,x^k)\in\vI^k$.

 For $s\in [0,1]$, define the map $R_s:\overrightarrow{[0,n]}\times \vI \to \vI$ by the formula $R_s(t,h)=sR(t,h)+(1-s)h$. The collection of maps $R^K_s:\overrightarrow{[0,n]} \times |K|\to |K|$ induced by the maps $R_s$ is a homotopy between $R^K$ and the projection on the second factor.

These maps define, for $0\leq a\leq b\leq n$, the following self-maps of d-path spaces:
\begin{align*}
	\bar{R}:\vP_{[a,b]}(\vI)&\to \vP_{[a,b]}(\vI)\\
	\bar{R}^k:\vP_{[a,b]}(\vI^k)&\to \vP_{[a,b]}(\vI^k)\\
	\bar{R}^K:\vP_{[a,b]}(K)&\to \vP_{[a,b]}(K)
\end{align*}
such that $\bar{R}(\alpha)(t)=R(t,\alpha(t))$, $\bar{R}^k(\alpha)(t)=R^k(t,\alpha(t))$, $\bar{R}^K(\alpha)(t)=R^K(t,\alpha(t))$. All these maps are homotopic to the respective identities via the families of maps $\bar{R}_s$, $\bar{R}^k_s$ and $\bar{R}^K_s$ that are defined in a similar way.

Now assume that $R:\overrightarrow{[0,n]}\times \vI \to \vI$ satisfies the following conditions:
\begin{enumerate}[\normalfont (a)]
\item $R(t,h)=0$ for $h\in [0,\tfrac{1}{4}]$,
\item $R(t,h)=1$ for $h\in [\tfrac{3}{4},1]$,
\item For every $h\in[0,1]$, the support of $R(-,h)$,
\[
	\supp(R(-,h))=\{t\in [0,n]\;|\; 0<R(t,h)<1\}
\]
is an interval having length less or equal to $\tfrac{1}{4n}$.
\end{enumerate}
Such a function exists; an example is given by the formula
\[
    R(t,h)=\min(1, \max(0, (4nt+12n^2h-8n^2))).
\]
These assumptions imply the following properties of $\bar{R}$:
\begin{prp}\label{p:CompressedSupport}
	For every $\alpha\in\vP_{[0,n]}(\vI)_0^1$, the support of $\bar{R}(\alpha)$ is an open interval having length less or equal to $\tfrac{1}{4n}$.
\end{prp}
\begin{proof}
	Obviously $\bar{R}(\alpha)(0)=R(0,\alpha(0))=0$ and $\bar{R}(\alpha)(n)=R(n,\alpha(n))=1$.
	Assume that $s<t\in\supp(\bar{R})(\alpha)$, i.e., $0<\bar{R}(\alpha)(s)\leq\bar{R}(\alpha)(t)<1$. Then
	\[
		0<\bar{R}(\alpha)(s)=R(s,\alpha(s))\leq R(t,\alpha(s)) \leq R(t,\alpha(t))=\bar{R}(\alpha)(t)<1.
	\]
	Hence $s,t\in \supp(R(-,\alpha(s)))$, which implies that $t-s<\tfrac{1}{4n}$.
\end{proof}

\begin{prp}\label{p:LongSupport}
	Assume that $\alpha\in \vP_{[0,n]}(\vI)_0^1$ is a 1-Lipschitz function. If $t\in \supp(\bar{R}(\alpha))$, then $[t-\tfrac{1}{4},t+\tfrac{1}{4}]\subseteq \supp(\alpha)$.
\end{prp}
\begin{proof}
	Assume that $\bar{R}(\alpha(t))=R(t,\alpha(t))>0$. By condition (a), we have $\alpha(t)>\tfrac{1}{4}$, which implies that $\alpha(t-\tfrac{1}{4})>0$. Using condition (b) we obtain that $\alpha(t)<\tfrac{3}{4}$ and, therefore, $\alpha(t+\tfrac{1}{4})<1$.
\end{proof}

\begin{prp}\label{p:TamingLemma}
	If $\alpha\in \vN_{[0,n]}(K)_\bO^\bI$, then $\bar{R}^K(\alpha)$ is tame.
\end{prp}
\begin{proof}
	Denote $\omega=\bar{R}^K(\alpha)$.
	Choose a track $\fC=(c_i,A_i,B_i)_{i=1}^l$ containing $\alpha$, and a $\fC$--presentation
	\begin{equation}\label{e:PresOfAlpha}
		\alpha=\overset{t_{0}}{\phantom{*}}[c_1;\beta_1]\overset{t_1}*\dots\overset{t_{l-1}}*[c_l;\beta_l] \overset{t_{l}}{\phantom{*}},
	\end{equation}
	Immediately from the definition follows that
	\begin{equation}
		\omega(t)=[c_i; R^{\dim(c_i)}(t, \beta_i(t))]
	\end{equation}
	for all $i\in\{1,\dots,l\}$, $t\in[t_{i-1},t_i]$. Thus,
	\begin{equation}\label{e:PresOfOmega}
		\omega=\overset{t_{0}}{\phantom{*}}[c_1;\bar{R}^{\dim(c_1)}(\beta_1)]\overset{t_1}*\dots\overset{t_{l-1}}*[c_l;\bar{R}^{\dim(c_l)}(\beta_l)]\overset{t_{l}}{\phantom{*}}
	\end{equation}
	is a $\fC$--presentation of $\omega$.
	Let $\bbf_\alpha=(f_\alpha^p)_{p\in T(\fC)}$ and $\bbf_\omega=(f_\omega^p)_{p\in T(\fC)}$ be the progress functions of $\alpha$ and $\omega$ associated to the presentations (\ref{e:PresOfAlpha}) and (\ref{e:PresOfOmega}), respectively. Clearly $f_\omega^p=\bar{R}(f_\alpha^p)$ for all $p\in T(\fC)$.
	Recall that
		\[
		\supp(\bbf_\omega)=\bigcup_{p\in T(\fC)}\supp(f_\omega^p) =\bigcup_{p\in T(\fC)}(a^p_{\bbf_\omega},b^p_{\bbf_\omega}).
	\]
	Let
	\[
		\supp(\bbf_\omega)=(x_1,y_1)\cup\dots\cup (x_r,y_r) \subseteq (0,n)
	\]
	be the decomposition into the union of connected components, ordered increasingly. Let $y_0=0$, $x_{r+1}=n$. 
	For every $s\in \{1,\dots,r\}$, let $A_s\subseteq T(\fC)$ be the set of actions $p$ such that $\supp(f_\omega^p)\subseteq (x_s,y_s)$.
	Clearly,
\[
	T(\fC)=A_1\cupdot\dots\cupdot A_r.
\]	
	 By Proposition \ref{p:CompressedSupport},  for every $s\in\{1,\dots,r\}$ the length of the interval $(x_s,y_s)$ is less than $\tfrac{1}{4n}|A_s|\leq\tfrac{1}{4}$. Thus, by Proposition \ref{p:LongSupport} we have
	\[
		(x_s,y_s) \subseteq \bigcap_{p\in A_s} \supp(f_\alpha^p),
	\]
	since $f^p_\alpha$ is a 1--Lipschitz function for every action $p$.
	
	Fix $s\in \{1,\dots,r\}$. Choose $i(s)\in\{1,\dots,l\}$ and $u_s\in[0,n]$ such that $u_s\in [t_{i(s)-1},t_{i(s)}]\cap (x_s,y_s)$. Assume that $p\in T^1_{i(s)}(\fC)$. For every $s'\geq s$ we have
	\[
		1=f^p_\alpha(t_{i(s)-1})\leq f^p_\alpha(u_{s})\leq f^p_\alpha(u_{s'}).
	\]
	Therefore, $u_{s'}\not\in\supp(f^p_\alpha)$, which implies that $p\not\in A_{s'}$. As a consequence,
	\[
		p\in A_1\cup \dots \cup A_{s-1}
	\]
	and then $b^p_{\bbf_\omega}\leq y_{s-1}$.
	Therefore, for all $s$ we have
	\[
		a^{\bbf_\omega}_{i(s)}=\max_{p\in T^1_{i(s)}(\fC)} b^{\bbf_\omega}_p \leq y_{s-1}.
	\]
	A similar argument shows that also
	\[
		b^{\bbf_\omega}_{i(s)}=\min_{p\in T^0_{i(s)}(\fC)} a_{\bbf_\omega}^p \geq x_{s+1}.
	\]
	Choose a sequence $(z_s)_{s=0}^{r}$ such that $z_0=0$, $z_{r}=n$ and $y_s\leq z_s\leq x_{s+1}$ for $s\in\{1,\dots,r-1\}$. For every $s\in \{0,\dots,r\}$, $\omega(z_s)$ is a vertex.  Since $[z_{s-1},z_s] \subseteq [y_{s-1},x_{s+1}]\subseteq [a^{\bbf_\omega}_{i(s)}, b^{\bbf_\omega}_{i(s)}]$, we have
	\[
		\omega|_{[z_{s-1},z_s]} \overset{\text{Prop. \ref{p:ProgressOfDPath}}}{=} \alpha^{\bbf_\omega}|_{[z_{s-1},z_s]}= \alpha^{\bbf_\omega}_{i(s)}|_{[z_{s-1},z_s]} =[c_{i(s)}; \beta^{\bbf^\omega}_{i(s)}|_{[z_{s-1},z_s]}],
	\]
	which shows that $\omega$ is tame.
\end{proof}

\begin{prp}\label{p:Qt}
	For every $s\in [0,1]$,  $\bar{R}^K_s(\vP^t_{[0,n]}(K)_\bO^\bI)\subseteq \vP^t_{[0,n](K)_\bO^\bI}$.
\end{prp}
\begin{proof}
	If $\alpha=\overset{t_{0}}{\phantom{*}}[c_1;\beta_1]\overset{t_{1}}*\dots\overset{t_{l-1}}*[c_l;\beta_l]\overset{t_{l}}{\phantom{*}}$ is a tame presentation, then also
	\[
		\bar{R}^K_s(\alpha)=[c_1; \bar{R}^{\dim(c_1)}_s(\beta_1)]*\dots*[c_r; \bar{R}^{\dim(c_r)}_s(\beta_r)]
	\]
	is a tame presentation, since $\bar{R}^{\dim(c_{i})}_s(\beta_{i})(t_{i-1})=\bO$ and $\bar{R}^{\dim(c_i)}_s(\beta_{i})(t_i)=\bI$.
\end{proof}

\begin{proof}[Proof of \ref{t:Tame}]
	By Proposition \ref{p:TamingLemma}, the diagram (\ref{e:TamingDiagram}) can be completed to the diagram
	\[
		\begin{diagram}
			\node{\vP^t_{[0,n]}(K;n)_\bO^\bI}
				\arrow{e,t}{\subseteq}
			\node{\vP_{[0,n]}(K;n)_\bO^\bI}
		\\
			\node{\vN^t_{[0,n]}(K)_\bO^\bI}
				\arrow{e,t}{\subseteq}
				\arrow{n,lr}{\subseteq}{\simeq}
			\node{\vN_{[0,n]}(K)_\bO^\bI}		
				\arrow{n,lr}{\simeq}{\subseteq}
				\arrow{nw,t}{\bar{R}^K}
		\end{diagram}
	\]
	Since $\bar{R}^K:\vP_{[0,n]}(K;n)_\bO^\bI\to \vP_{[0,n]}(K;n)_\bO^\bI$ is homotopic to the identity via the maps $\bar{R}^K_s$, the right-hand triangle of the diagram commutes up to homotopy this diagram commutes up to homotopy. For similar reasons, Proposition \ref{p:Qt} implies that also the left-hand triangle commutes up to homotopy. Both vertical inclusions are homotopy equivalences: their homotopy inverses are the naturalization maps (see Subsection \ref{ss:Nat}), since the a naturalization of a tame d-path is tame. Now an easy  diagram-chasing argument shows that all the maps in the diagram are homotopy equivalences.
\end{proof}

Let $\Tam_n^K:\vP_{[0,n]}(K;n)_\bO^\bI \to \vN^t_{[0,n]}(K)_\bO^\bI$ be the composition
\begin{equation}\label{e:Tam}
		\vP_{[0,n]}(K;n)_\bO^\bI
			\xrightarrow{\nat^K_n}
		\vN_{[0,n]}(K)_\bO^\bI
			\xrightarrow{\bar{R}^K}
		\vP^t_{[0,n]}(K;n)_\bO^\bI
			\xrightarrow{\nat^K_n}
		\vN^t_{[0,n]}(K)_\bO^\bI.
\end{equation}
Here follows an immediate consequence of Theorem \ref{t:Tame}:

\begin{cor}\label{c:Tam}
	For every $n\geq 0$ and every bi-pointed $\square$--set $K\in \square\Set\bip$, the map $\Tam^K_n$ a homotopy inverse of the inclusion $i^K_n:\vN^t_{[0,n]}(K)_\bO^\bI\subseteq \vP_{[0,n]}(K;n)_\bO^\bI$. Furthermore,  the maps $\Tam^K_n$ are functorial with respect to $K$, i.e., they define the natural transformation $\vP_{[0,n]}(-;n)_\bO^\bI\Rightarrow \vN^t_{[0,n]}(-)_\bO^\bI$ of functors $\square\Set\bip\to\Top$.
\end{cor}

\section{Cube chains}

In this section we define the cube chain category $\Ch(K)$ of a bi-pointed $\square$--set $K$, and  formulate the main results of the paper.

\begin{df}
	Let $(K,\bO_K,\bI_K)$, $(L,\bO_L,\bI_L)$ be bi-pointed $\square$--sets. \emph{The wedge} of $K$ and $L$ is a bi-pointed $\square$--set $(K\vee L,\bO_{K\vee L}, \bI_{K\vee L)}$ such that
	\begin{itemize}
	\item{$(K\vee L)[n]=K[n]\coprod L[n]$ for $n>0$,}
	\item{$(K\vee L)[0]=(K[0]\coprod L[0])/\bI_K\sim \bO_L$,}
	\item{the face maps of $K\vee L$ are the disjoints unions of the face maps of $K$ and $L$,}
	\item{$\bO_{K\vee L}=\bO_K$, $\bI_{K\vee L}=\bI_L$.}
	\end{itemize}
\end{df}

\begin{rem}
	The wedge operation is associative but not commutative.
\end{rem}

Let $\Seq(n)$ be the set of sequences of positive integers $\bn=(n_1,\dots,n_l)$ such that $n_1+\dots+n_l=n$. For $\bn\in\Seq(n)$,
\begin{itemize}
	\item{$|\bn|=n$ is \emph{the length} of $\bn$,}
	\item{$l(\bn)=l $ is the number of elements of $\bn$,}
	\item{$t^\bn_i=\sum_{j=1}^i n_j$ for $i\in\{0,\dots,l(\bn)\}$ is \emph{the $i$--th vertex of $\bn$},}
	\item{$\Vertic(\bn)=\{t^\bn_i\}_{i=0}^{l(\bn)}$ is the set of \emph{vertices} of $\bn$,}
	\item{$\Free(\bn)=\{1,\dots,n-1\}\setminus \Vertic(\bn)$.}
\end{itemize}

\begin{df}
	Let $\bn\in \Seq(n)$. \emph{The wedge $\bn$-cube} is the bi-pointed $\square$--set
	\begin{equation*}
		\square^{{\vee}\bn}:= \square^{n_1}\vee\square^{n_2}\vee\dots\vee \square^{n_l},
	\end{equation*}
	where $\square^{n_i}$ is regarded as a bi-pointed $\square$--set by taking $\bO=(0,\dots,0)$, $\bI=(1,\dots,1)$. The geometric realization of $\square^{{\vee}\bn}$ will be denoted by $\vI^{{\vee}\bn}$. For $i\in\{1,\dots,l(\bn)\}$,
	\[
		v^\bn_i=\bI_{\square^{n_{i-1}}}=\bO_{\square^{n_i}} \in \square^{{\vee}\bn}[0]
	\]
	 is  \emph{the $i$--th vertex} of $\square^{{\vee}\bn}$.
\end{df}

Let $\square\mathbf{Ch}_n$ be the full subcategory of $\square\Set\bip$ with objects $\square^{{\vee}\bn}$ for $\bn\in\Seq(n)$.
We will introduce a notation for some morphisms of $\square\mathbf{Ch}_n$.
For a partition
\[
	\{1,\dots,m_1+m_2\}=A\cupdot B, \qquad |A|=m_2>0,\quad |B|=m_1>0,
\]
 let $\varphi_{A,B}:\square^{m_1}\vee \square^{m_2} \to \square^{m_1+m_2}$ be the unique bi-pointed $\square$--map such that $\varphi_{A,B}(u_{m_1})=d^0_{B}(u_{m_1+m_2})$ and $\varphi_{A,B}(u_{m_2})=d^1_A(u_{m_1+m_2})$. For $\bn\in \Seq(n)$, $i\in \{1,\dots,l(\bn)\}$ and $A\cupdot B=\{1,\dots,n_i\}$ let
\begin{multline}\label{e:DeltaWedgeCubeMap}
	\delta_{i,A,B}: \square^{n_1} \vee \dots\vee \square^{n_{i-1}} \vee\square^{|A|} \vee \square^{|B|} \vee\square^{n_{i+1}} \vee \dots \square^{n_l}\\
	\xrightarrow{\id_{\square^{n_1}}\vee\dots\vee \id_{\square^{n_{i-1}}}\vee \varphi_{A,B}\vee \id_{\square^{n_{i+1}}}\vee\dots\vee \id_{\square^{n_{l}}}}
	\square^{n_1} \vee \dots\vee \square^{n_{i-1}} \vee\square^{n_i}\vee\square^{n_{i+1}} \vee \dots \square^{n_l}=\square^{{\vee}\bn}.
\end{multline}

Fix a bi-pointed $\square$--set  $K\in\square\Set\bip$.
\begin{df}
	\emph{A cube chain} in $K$ (of length $n$) is a bi-pointed $\square$--map $\bc:\square^{{\vee}\bn^\bc}\to K$ for some $\bn^\bc\in\Seq(n)$. The multi-index $\bn=\bn^\bc$ will be called \emph{the type} of $\bc$. For short we denote $l(\bc)=l(\bn^\bc)$, $t^\bc_i=t^{\bn_\bc}_i$.
\end{df}

There is a 1-1 correspondence between cube chains in $K$ and sequences of cubes $(c_1,\dots,c_l)$ such that $d^0(c_1)=\bO_K$, $d^1(c_l)=\bI_K$ and $d^1(c_i)=d^0(c_{i+1})$ for all $i\in\{1,\dots,l-1\}$. Indeed, a cube chain $\bc:\square^{{\vee}\bn}\to K$ determines the sequence $(\bc(u_{\dim(c_i)}))_{i=1}^l$, and a sequence $(c_i)_{i=1}^l$ determines the unique $\square$--map
\begin{equation}
	\bc: \square^{\dim(c_1)}\vee \dots\vee \square^{\dim(c_l)} \to K
\end{equation}
such that $\bc(u_{\dim(c_i)})=c_i$. This shows that the notion of a cube chain defined here coincides with the definition introduced in  \cite{Z-Perm}. Below, we will identify cube chains $\bc$ and the corresponding sequences $(c_1,\dots,c_l)$ satisfying the conditions above.


\begin{df}
	\emph{The length $n$ cube chain category} $\Ch(K;n)$ of $K$ is the slice category $\square\mathbf{Ch}_n\downarrow K$. In other words, objects of $\Ch(K;n)$ are cube chains of length $n$, and morphisms from $\ba$ to $\bb$ are commutative diagrams
	\begin{equation}
		\begin{diagram}
			\node{\square^{{\vee}\bn^\ba}}
				\arrow{e,t}{\ba}
				\arrow{s,l}{\varphi}
			\node{K}	
				\arrow{s,r}{=}
		\\
			\node{\square^{{\vee}\bn^\bb}}
				\arrow{e,t}{\bb}
			\node{K}	
		\end{diagram}
	\end{equation}
	in $\square\Set\bip$. \emph{The full cube chain category} of $K$ is $\Ch(K):=\coprod_{n\geq 0} \Ch(K;n)$.
\end{df}

For $\bc\in \Ch(K;n)$, $i\in\{1,\dots,l(\bn^\bc)\}$ and $A\cupdot B=\{1,\dots,n^\bc_i\}$ denote $d_{i,A,B}(\bc)=\bc\circ \delta_{i,A,B}$. The morphism $d_{i,A,B}(\bc)\to \bc$ of $\Ch(K;n)$ that is given by the composition with $\delta_{i,A,B}$ will be also denoted by  $\delta_{i,A,B}$.

There is a forgetful functor $\dom: \Ch(K;n)\to \square\Set\bip$ that assigns to every cube chain $\bc:\square^{{\vee}\bn^\bc}\to K$ its domain $\square^{{\vee}\bn^\bc}$. This is equipped with the natural transformation $\dom\to \const_K$, which is induced by the chains themselves. For every $n\geq 0$, this transformation induces the map
\begin{equation}\label{e:ChainCoveringMap}
	F_n^K:\Lcolim_{\ba\in \Ch_n(K;n)_v^w}  \vN_{[0,n]} (\square^{{\vee}\bn^\ba})_{\bO}^{\bI}\to \vN^{t}_{[0,n]}(K)_\bO^\bI.
\end{equation}

Consider the sequence of maps
\begin{equation} \label{e:MainSequence}
	|\Ch(K;n)| \xleftarrow{A_n^K} \Lhocolim_{\bc\in \Ch(K;n)}{\vN(\square^{{\vee}\bn^\bc})_{\bO}^\bI}
	\xrightarrow{Q_n^K} {\Lcolim_{\bc\in \Ch(K;n)}  {\vN(\square^{{\vee}\bn^\bc})_{\bO}^\bI}}
	\xrightarrow{F_n^K} {\vN_{[0,n]}^t(K)_{\bO}^\bI},
\end{equation}
where $A_n^K$ is the composition
\[
	 \Lhocolim_{\bc\in \Ch(K;n)}{\vN(\square^{{\vee}\bn^\bc})_{\bO}^\bI} \to  \Lhocolim_{\bc\in \Ch(K;n)}\{*\}\cong |\Ch(K;n)|,
\]
and $Q_n^K$ is the natural map from the homotopy colimit to the colimit of the functor $\vN(\square^{{\vee}\bn^{(-)}})_{\bO}^\bI$.

\begin{thm}\label{t:ChN}
	For every $K\in\square\Set\bip$, all the maps in the sequence {\normalfont (\ref{e:MainSequence})} are weak homotopy equivalences.
\end{thm}
\begin{proof}
	From \cite[Proposition 6.2.(1)]{Z-Perm} follows that the space $\vN_{[0,n]}(\square^{{\vee}\bn})_\bO^\bI$ is contractible for every $\bn\in\Seq(n)$. As a consequence, the map $A^K_n$ is a weak homotopy equivalence by \cite[Proposition 4.7]{Dugger}. The maps $Q^K_n$ and $F^K_n$ are homotopy equivalences by Proposition \ref{p:QIsHtpEq} and Proposition \ref{p:ColimNHEq}, respectively.
\end{proof}

As an immediate consequence, we obtain the main result of this paper.

\begin{thm}\label{t:Main}
	The functors
	\[
		\square\Set\bip\ni (K,\bO,\bI) \mapsto \vP(K)_\bO^\bI\in \hTop
	\]
	and
	\[
		\square\Set\bip\ni (K,\bO,\bI) \mapsto |\Ch(K)|\in \hTop
	\]
	are naturally equivalent. In particular, for every bi-pointed $\square$--set $K$, the spaces $\vP(K)_\bO^\bI$ and $\lvert \Ch(K)\rvert$ are homotopy equivalent.
\end{thm}
\begin{proof}
	For every $K\in\square\Set\bip$, there is a sequence of maps
	\[
		\lvert\Ch(K)\rvert
			=
		\coprod_{n\geq 0} \lvert\Ch(K;n)\rvert
			\underset{\text{Thm.\ref{t:ChN}}}{\xrightarrow{\hspace*{1em}\simeq\hspace*{1em}}}
		\coprod_{n\geq 0} \vN^t_{[0,n]}(K;n)_\bO^\bI 
			\underset{\text{Thm.\ref{t:Tame}}}{\xrightarrow{\hspace*{1em}\subseteq\hspace*{1em}}}
		 \coprod_{n\geq 0} \vP(K;n)_\bO^\bI
		 	\underset{\text{(\ref{e:LengthDecomposition})}}=
		  \vP(K)_\bO^\bI
	\]
	which are all functorial weak homotopy equivalences.
\end{proof}

\section{Natural tame presentations}

In this section we study properties of presentations of natural tame paths on an arbitrary bi-pointed $\square$--set $K$. Fix an integer $n\geq 0$, which will be the length of all d-paths considered here.

If
\begin{equation}\label{e:NatTamPres}
	\alpha=\overset{t_0}{\phantom{*}}[c_1;\beta_1]\overset{t_1}{*} [c_2;\beta_2]\overset{t_2}{*}\dots\overset{t_{l-1}}{*}[c_{l};\beta_{l}]\overset{t_{l}}{\phantom{*}},
\end{equation}
is a tame presentation of a tame natural d-path $\alpha\in \vN^t_{[0,n]}(K)_{\bO}^{\bI}$, then $\bc=(c_1,\dots,c_l)$ is a cube chain in $K$, and
\begin{equation}
	\beta=\beta_1*\beta_2*\dots*\beta_l\in \vN_{[0,n]}(\vI^{{\vee}\bn^\bc})_\bO^\bI
\end{equation}
is a natural d-path such that $\alpha=[\bc;\beta]=|\bc|\circ \beta$. For every pair $(\bc,\beta)$  such that $\bc\in\Ch(K;n)$ and $\beta\in\vN_{[0,n]}(K)_\bO^\bI$ we can associate the presentation (\ref{e:NatTamPres}) when we take $\beta_i=\beta|_{[t^\bc_{i-1},t^\bc_i]}$, regarded as a natural d-path in $\vI^{n^\bc_i}\subseteq \vI^{{\vee}\bn^\bc}$. This is construction is correct: since $\beta(t^\bc_{i})=v^{\bn^\bc}_i$, the whole segment $\beta_i$ lies in the $i$--th cube of the wedge cube $\vI^{{\vee}\bn^\bc}$. This identifies tame presentations of natural paths having length $n$ with elements of $\coprod_{\bc\in\Ch(K;n)} \vN_{[0,n]}(\vI^{{\vee}\bn^\bc})_\bO^\bI$ that are mapped into $\alpha$ by the composition
\begin{equation}
	\coprod_{\bc\in\Ch(K;n)} \vN_{[0,n]}(\vI^{{\vee}\bn^\bc})_\bO^\bI \to \Lcolim_{\bc\in\Ch(K;n)} \vN_{[0,n]}(\vI^{{\vee}\bn^\bc})_\bO^\bI \xrightarrow{F^n_K} \vN^t_{[0,n]}(K)_{\bO}^{\bI}.
\end{equation}
Thus, elements of $\coprod_{\bc\in \Ch(K;n)} \vN_{[0,n]}(\vI^{\bn^\bc})_\bO^\bI$ will be also called natural tame presentations, or nt-presentations for short.

For any d-path $\alpha\in \vP_{[a,b]}(K)$ denote
\begin{equation}
	\Vertic(\alpha)=\alpha^{-1}(|K_{(0)}|)= \{t\in [a,b]\;|\; \text{$\alpha(t)$ is a vertex}\}.
\end{equation}
Note that if $\alpha\in \vN_{[0,n]}(K)_{\bO}^\bI$, then $\Vertic(\alpha)\subseteq \{0,1,\dots,n\}$. If additionally $\alpha$ admits a natural tame presentation $\alpha=[\bc;\beta]$, then $ \Vertic(\bn^\bc) \subseteq \Vertic(\alpha)$.

\begin{df}
	A natural tame presentation $\alpha=[\bc;\beta]$ is
	\begin{enumerate}[\normalfont (a)]
	\item{
		\emph{minimal} if $\Vertic(\alpha)=\Vertic(\bn^\bc)$,
	}
	\item{
		\emph{regular} if for every $i\in\{1,\dots,l(\bc)\}$ there exists $t\in[t^\bc_{i-1},t^\bc_{i}]$ such that $\beta_i(t)\in (0,1)^{n^\bc_i}$,
	}
	\item{
		\emph{equivalent}	 to a natural tame presentation $\alpha=[\bc';\beta']$ if both $(\bc,\beta)$ and $(\bc',\beta')$ represent the same element in $\colim_{\bc\in\Ch(K;n)} \vN_{[0,n]}(\vI^{{\vee}\bn^\bc})_\bO^\bI$.
	}
	\end{enumerate}
\end{df}
Equivalently, $\alpha=[\bc;\beta]$ is not regular if $\beta_i\in\vN_{[t^\bc_{i-1},t^\bc_{i}]}(\partial\vI^{n^\bc_i})$ for some $i\in\{1,\dots,l(\bc)\}$.

The set of equivalence classes of nt-presentations of $\alpha$ is just $(F^n_K)^{-1}(\alpha)$. The example below shows that there may exist non-equivalent nt-presentations of given natural tame path.
\begin{exa}\label{x:TameExotic}
	Let $K=\square^3 \cup_{\partial\square^3} \square^3$ be the union of two standard $3$--cubes, glued along their boundaries. Denote the 3--cubes of $K$ by $c$ and $c'$ and choose a d-path $\alpha\in\vN(\partial\vI^3)_{\bO}^\bI$ that is not tame. But $\alpha$ regarded as a d-path in $K$ is tame, and has two different tame presentations: $[c;\alpha]$ and $[c';\alpha]$, which are not equivalent.
\end{exa}

\begin{prp}\label{p:AddingVerticesToNt}
	Assume that $\alpha=[\bc;\beta]$ is an nt-presentation. Assume that $k\in \Vertic(\alpha)\setminus\Vertic(\bn^\bc)$. Then there exists an nt-presentation $\alpha=[\bc';\beta']$ that is equivalent to $[\bc;\beta]$ such that $\Vertic(\bn^{\bc'})=\Vertic(\bn^\bc)\cup \{k\}$.
\end{prp}
\begin{proof}
	Let $i\in\{1,\dots,l(\bc)\}$ be an integer such that $t^{\bc}_{i-1}<k<t^{\bc}_i$. We have $\beta_i(k)\in\{0,1\}^{n^\bc_i}$; denote
	\[
		A=\{j\in\{1,\dots,n^\bc_i\}\;|\; \beta_i^j(k)=1\},\quad B=\{j\in\{1,\dots,n^\bc_i\}\;|\; \beta_i^j(k)=0\}.
	\]
	Obviously $\beta^j_i|_{[t^\bc_{i-1},k]}\equiv 0$ for all $j\in B$ and  $\beta^j_i|_{[k,t^\bc_{i}]}\equiv 1$ for all $j\in A$. Therefore, there exist unique paths $\gamma\in \vN_{[t^\bc_{i-1},k]}(\vI^{|A|})_\bO^\bI$, $\gamma'\in \vN_{[k,t^\bc_{i}]}(\vI^{|B|})_\bO^\bI$ such that $\beta_i|_{[t^\bc_{i-1},k]}= \delta^0_B(\gamma)$, $\beta_i|_{[k,t^\bc_{i}]}= \delta^1_A(\gamma')$. As a consequence,
	\[
		[c_i;\beta_i]=[c_i;\delta^0_B(\gamma)]*[c_i;\delta^1_A(\gamma')]=[d^0_B(c_i);\gamma]*[d^1_A(c_i);\gamma'].
	\]
	Let
	\begin{align*}
		\bc'&=(c_1,\dots,c_{i-1},d^0_B(c_i),d^1_A(c_i),c_{i+1},\dots,c_{l(\bc)})\in \Ch(K;n),\\
		\beta'&=\beta_1*\dots*\beta_{i-1}*\gamma*\gamma'*\beta_{i+1}*\dots*\beta_{l(\bc)}\in \vN_{[0,n]}(\vI^{{\vee}\bn^{\bc'}})_\bO^\bI.
	\end{align*}
	We have $d_{i,A,B}(\bc)=\bc'$. The map
	\[
		\vN_{[0,n]}(\delta_{i,A,B})_\bO^\bI: \vN_{[0,n]}(\vI^{{\vee}\bn^{\bc'}}) \to \vN_{[0,n]}(\vI^{{\vee}\bn^{\bc}})
	\]
	induced by the morphism $\delta_{i,A,B}:\bc'\to \bc$ sends $\beta$ into $\beta'$. As a consequence, the presentations $[\bc;\beta]$ and $[\bc',\beta']$ are equivalent. The condition $\Vertic(\bn^{\bc'})=\Vertic(\bn^\bc)\cup \{k\}$ follows immediately from the definition of $\bc'$.
\end{proof}

\begin{prp}\label{p:ExistenceOfMinimalPresentation}
	Every natural tame presentation is equivalent to a minimal natural tame presentation.
\end{prp}
\begin{proof}
	This follows by induction from Proposition \ref{p:AddingVerticesToNt}.
\end{proof}

\begin{prp}\label{p:RegularIsMinimal}
	Every regular nt-presentation is minimal.
\end{prp}
\begin{proof}
	Assume that $[\bc,\beta]$ is a regular nt-presentation. Then, for every $i\in\{1,\dots,l(\bc)\}$ there exists $t^{\bn^\bc}_{i-1}<t<t^{\bn^\bc}_{i}$ such that $\beta_i(t)\in (0,1)^{\dim(c_i)}$, which implies that $\Vertic(\beta_i)=\{t^{\bn^\bc}_{i-1}, t^{\bn^\bc}_i\}$. Since $\Vertic(\alpha)=\bigcup_i\Vertic(\beta_i)$, the presentation $[\bc;\beta]$ is minimal.
\end{proof}

A d-path $\alpha\in\vN^t_{[0,n]}(K)_\bO^\bI$ is \emph{regular} if it admits a regular presentation. To prove that all nt-presentations of regular d-paths are equivalent, we need the following.

\begin{lem}\label{l:CubePathPresentations}
	Let $\alpha,\beta\in \vP_{[a,b]}(\vI^m)$ be d-paths such that:
	\begin{enumerate}[\normalfont (a)]
	\item{$\alpha(a)=\beta(a)=\bO$,}
	\item{$\alpha(b)=\beta(b)\not\in \partial\vI^n$, i.e., $0<\alpha^j(b)=\beta^j(b)<1$ for all $j$,}
	\item{for every $t\in[a,b]$, the sequences $(\alpha^1(t),\dots,\alpha^m(t))$ and $(\beta^1(t),\dots,\beta^m(t))$ are equal after removing all $0$'s.}
	\end{enumerate}
	Then $\alpha=\beta$.
\end{lem}
\begin{proof}
	Induction with respect to $m$. For $m=1$ this is obvious.
	Let $s\in[a,b]$ be the maximal number such that $\alpha^j(s)=0$ for some $j\in\{1,\dots,n\}$. Condition (c) implies that  $\alpha(t)=\beta(t)$ for $t>s$, since there are no $0$'s to remove. By continuity also $\alpha(s)=\beta(s)$. If $s=0$, then the lemma is proven. Assume $s>0$ and let $A=\{j\;|\; \alpha^j(s)=\beta^j(s)=0\}$. Let $\alpha',\beta'\in \vP_{[a,s]}(\vI^{m-|A|})$ be the paths obtained from $\alpha|_{[a,s]}$ and $\beta|_{[a,s]}$ by removing all coordinates belonging to $A$. By the inductive hypothesis, we have $\alpha'=\beta'$ and, therefore, $\alpha|_{[a,s]}=\beta|_{[a,s]}$.
\end{proof}

\begin{prp}\label{p:RegularNTAreEquivalent}
	Every regular natural tame d-path $\alpha\in\vN^t_{[0,n]}(K)_\bO^\bI$ admits a unique minimal nt-presentation, which is regular. As a consequence, all nt-presentations of $\alpha$ are equivalent.
\end{prp}
\begin{proof}
	Let $\alpha=[\bc;\beta]$ be a regular nt-presentation and let $\alpha=[\bc';\beta']$ be a minimal presentation. By Proposition \ref{p:RegularIsMinimal}, we have $\Vertic(\bn^\bc)=\Vertic(\alpha)=\Vertic(\bn^{\bc'})$ and hence $\bn^\bc=\bn^{\bc'}$. Fix $i\in\{1,\dots,l(\bc)=l(\bc')\}$ and choose $t\in [t^\bc_{i-1},t^\bc_i]$ such that $\beta_i(t)\in \vI^{n^\bc_i}\setminus \partial\vI^{n^\bc_i}$. Now $[c_i;\beta_i(t)]$ is a canonical presentation of the point $\alpha(i)$. Since $\dim(c'_i)=\dim(c_i)$, also $[c'_i,\beta'_i(t)]$ is a canonical presentation of $\alpha(t)$, which implies that $c_i=c'_i$ and $\beta_i(t)=\beta'_i(t)$. Applying Lemma \ref{l:CubePathPresentations} we obtain that $\beta_i|_{[t^\bc_{i-1},t]}=\beta'_i|_{[t^\bc_{i-1},t]}$, and the ``opposite" analogue of this lemma implies that $\beta_i|_{[t, t^\bc_{i}]}=\beta'_i|_{[t, t^\bc_{i}]}$. As a consequence, $c_i=c'_i$ and $\beta_i=\beta'_i$ for all $i$, and hence $\bc=\bc'$ and $\beta=\beta'$. Thus, the minimal nt-presentation of $\alpha$ is unique and all nt-presentations of  $\alpha$ are equivalent to that.
\end{proof}

\begin{prp}\label{p:EquivalenceOfMinNTForCubes}
	Assume that $K\subseteq \square^n$. Then every natural tame path $\alpha\in \vN^t_{[0,n]}(K)_\bO^\bI$ admits a unique minimal presentation.
\end{prp}
\begin{proof}
	Let $\alpha=[\bc,\beta]$ and $\alpha=[\bc',\beta']$ be minimal nt-presentations. We have $\bn:=\bn^\bc=\bn^{\bc'}$. For every $i\in\{1,\dots,l(\bc)=l(\bc')\}$ we have
	\begin{align*}
		[d^0(c'_i);()]=[c'_i;\bO]=[c'_i,\beta'_i(t^\bn_{i-1})]&=\alpha(t^\bn_{i-1})=[c_i,\beta_i(t^\bn_{i-1})]=[c_i,\bO]=[d^0(c_i);()]\\
		[d^1(c'_i);()]=[c'_i;\bI]=[c'_i,\beta'_i(t^\bn_i)]&=\alpha(t^\bn_i)=[c_i,\beta_i(t^\bn_i)]=[c_i,\bI]=[d^1(c_i);()].
	\end{align*}
	Thus, $c_i=c'_i$ since they have the same extreme vertices. Furthermore, all the maps $|c_i|:\vI^{n_i}\to \vI^n$ are injective, which shows that $\beta_i=\beta'_i$.
\end{proof}

\section{Comparison of homotopy colimit and colimit}

Fix a bi-pointed $\square$--set $K$ and an integer $n\geq 0$.
In this section we show that the map
\begin{equation}\label{e:HocolimColim}
	Q_n^K:\Lhocolim_{\bc\in \Ch(K;n)}{\vN(\square^{{\vee}\bn^\bc})_{\bO}^\bI}
		\to
	{\Lcolim_{\bc\in \Ch(K;n)}  {\vN(\square^{{\vee}\bn^\bc})_{\bO}^\bI}}
\end{equation}
is a weak homotopy equivalence. We will use the following criterion due to Dugger \cite{Dugger}. Let $\cC$ be an upwards-directed Reedy category \cite[Definition 13.6]{Dugger} and let $F:\cC\to\Top$ be a diagram. \emph{The latching object} of $c\in\Ob(\cC)$ is
\begin{equation}
	L_cF:= \colim_{\partial \cC\downarrow c} F\circ \dom,
\end{equation}
where $\partial \cC\downarrow c$ is the slice category over the object $c$ with the identity object $\id_\bc$ removed, and $\dom:\cC\downarrow c\to \cC$ is the forgetful functor. The latching object is equipped with \emph{the latching map} $\Lambda_c:L_cF \to F(c)$ that is induced by the cocone $\{F(\varphi)\}_{(\varphi:a\to c)\in\partial\cC\downarrow c}$.

\begin{prp}[{\cite[Proposition 14.2]{Dugger}}]\label{p:Dugger}
	Assume that the latching map $\Lambda_c$ is a cofibration for every object $c\in\Ob(\cC)$. Then F is projective--cofibrant and so $\hocolim_\cC F \to  \colim_\cC F$ is a weak equivalence.
\end{prp}

The category $\Ch(K;n)$ admits a grading $\deg(\ba)=n-l(\ba)$, which makes it an upwards-directed Reedy category: for every non-identity morphism $\varphi:\ba\to \bb$ in $\Ch(K)$ we have $l(\ba)>l(\bb)$.

For every $\bc\in\Ch(K;n)$, the functor
\begin{equation}
	\Ch(\square^{{\vee}\bn^\bc})_\bO^\bI \ni (\square^{{\vee}\bn^\ba}\xrightarrow\ba \square^{{\vee}\bn^\bc})\mapsto
	\left(
		\begin{diagram}
			\node{\square^{{\vee}\bn^\ba}}
				\arrow{s,l}{\ba}
				\arrow{e,t}{\bc\circ \ba}
			\node{K}
				\arrow{s,r}{=}
		\\
			\node{\square^{{\vee}\bn^\bc}}
				\arrow{e,t}{\bc}
			\node{K}
		\end{diagram}
	\right)
	\in \Ch(K;n)\downarrow\bc
\end{equation}
is an isomorphism of categories: the functor $\dom:\Ch(K;n)\downarrow\bc\to \square\Ch_n$ is its inverse. Furthermore, the functors $\vN_{[0,n]}(\square^{{\vee}\bn^{(-)}})_\bO^\bI$ on $\Ch(\square^{{\vee}\bn^\bc})$ and $\vN_{[0,n]}(\square^{{\vee}\bn^{(-)}})_\bO^\bI \circ \dom$ on $\Ch(K;n)\downarrow\bc$ are naturally equivalent. As a consequence, the latching map $L_\bc \vN_{[0,n]}(\vI^{{\vee}\bn^{(-)}})\to \vN_{[0,n]}(\square^{{\vee}\bn^{\bc}})$ is homeomorphic to the canonical map
\begin{equation}
	\Lcolim_{\ba\in \partial\Ch(\square^{{\vee}\bn^\bc})} \vN_{[0,n]}(\vI^{{\vee}\bn^\ba})_\bO^\bI \to \vN_{[0,n]}(\vI^{{\vee}\bn^\bc})_\bO^\bI,
\end{equation}
where $\partial\Ch(\square^{{\vee}\bn^\bc})$ is the full subcategory of $\Ch(\square^{{\vee}\bn^\bc})$ with the identity chain removed. In particular, the latching map of $\bc$ depends only on the type $\bn^\bc$ of $\bc$.

For $\bn\in\Seq(n)$ denote
\begin{equation}
	\partial \vN_{[0,n]}(\square^{{\vee}\bn})_\bO^\bI = \{\alpha\in \vN_{[0,n]}(\square^{{\vee}\bn})_\bO^\bI \;|\; \Vertic(\alpha)\setminus\Vertic(\bn)\neq\emptyset \ \};
\end{equation}
this is the space of all natural tame paths that admit an nt-presentation $[\bc;\beta]$ such that $\bc$ is a non-identity cube chain.

\begin{prp}\label{p:BoundaryCover}
	For every $\bn\in \Seq(n)$ we have
	\[
		\im\left(
			\Lcolim_{\bc\in \partial\Ch(\square^{{\vee}\bn})} \vN_{[0,n]}(\vI^{{\vee}\bn^\bc})_\bO^\bI\xrightarrow{\Lambda_\bn} \vN_{[0,n]}(\vI^{{\vee}\bn})\right) = \partial \vN_{[0,n]}(\square^{{\vee}\bn})_\bO^\bI
	\]
	Furthermore, $\Lambda_\bn$ is a homeomorphism onto its image.
\end{prp}
\begin{proof}
	Assume that $\alpha\in \partial \vN_{[0,n]}(\square^{{\vee}\bn})_\bO^\bI$ and let $\alpha=[\bc;\beta]$ be a minimal nt-presentation, which exists due to Proposition \ref{p:ExistenceOfMinimalPresentation}. Since $\Vertic(\bn^\bc)=\Vertic(\alpha)\neq \Vertic(\bn)$, $\bc$ is not the identity cube chain. Therefore, $\Lambda_n(\bc,\beta)=\alpha$, which implies that $\partial \vN_{[0,n]}(\square^{{\vee}\bn})_\bO^\bI\subseteq \im(\Lambda_\bn)$.
	
If $\alpha\in\im(\Lambda_\bn)$, then there exists an nt-presentation $\alpha=[\bc,\beta]$ such that $\bc\neq \id_{\square^{{\vee}\bn}}$. Thus, $\Vertic(\alpha)\supseteq \Vertic(\bn^\bc)\supsetneq \Vertic(\bn)$ and, therefore, $\alpha\in \partial \vN_{[0,n]}(\square^{{\vee}\bn})_\bO^\bI$.

There exists an embedding of $\square^{{\vee}\bn}$ in $\square^n$. Then Proposition \ref{p:EquivalenceOfMinNTForCubes} implies that all nt-presentations of a d-path $\alpha\in \partial \vN_{[0,n]}(\square^{{\vee}\bn})_\bO^\bI$ are equivalent, since they are all equivalent to a unique minimal nt-presentation. As a consequence, $(\Lambda_\bn)^{-1}(\alpha)$ is a one-point set and, therefore, $\Lambda_\bn$ is a bijection onto its image.

By Proposition \ref{p:vNIsCompact}, all the spaces appearing in the colimit are compact. Since the indexing category $\partial\Ch(\square^{{\vee}\bn})$ is finite, the colimit is compact and $\Lambda_\bn$ is a homeomorphism onto its image.
\end{proof}

\begin{prp}\label{p:PartialNCofibration}
	For every $\bn\in\Seq(n)$, the inclusion $\partial \vN_{[0,n]}(\vI^{{\vee}\bn})_\bO^\bI\subset \vN_{[0,n]}(\vI^{{\vee}\bn})_\bO^\bI$ is a cofibration.
\end{prp}
\begin{proof}
	Denote for short $\partial\vN:= \partial \vN_{[0,n]}(\square^{{\vee}\bn})_\bO^\bI$, $\vN:=\vN_{[0,n]}(\square^{{\vee}\bn})_\bO^\bI$, $\vP:=\vP_{[0,n]}(\square^{{\vee}\bn})_\bO^\bI$. We will prove that $\partial \vN$ is a strong neighborhood deformation retract in $\vN$. If $\bn=(1,\dots,1)$, then $\Vertic(\bn)=\{0,1,\dots,n\}$ and hence $\partial \vN=\emptyset$. Assume that $\bn\neq(1,\dots,1)$ and fix an embedding $\square^{{\vee}\bn}\subseteq\square^n$. For any d-path $\alpha\in \vP_{[0,n]}(\vI^{{\vee}\bn})_\bO^\bI$, let $\alpha^j\in \vP_{[0,n]}(\vI)_0^1$ denote the $j$--th coordinate of $\alpha$, regarded as a d-path in $\vI^n$ (via the chosen embedding $\vI^{\bn}=|\square^{{\vee}\bn}|\subseteq |\square^n|=\vI^n$) .
	
	The function
	\[
		f: \vN \ni \alpha \mapsto \min_{k\in \Free(\bn)} \max_{j\in\{1,\dots,n\}} \min\{\alpha^j(k), 1-\alpha^j(k)\}\in [0,1]
	\]
	is continuous, and $f^{-1}(0)=\partial\vN$. Fix $0<\varepsilon<\tfrac{1}{n}$. For $a\in[0,\varepsilon]$ let $g(a,-):\vI\to \vI$ be the function such that $g(a,h)=0$ for $h\in [0,a]$, $g(a,t)=1$ for $h\in[1-a,1]$ and $g(a,-)$ is linear on the interval $[a,1-a]$. Note that $|t-g(a,t)|\leq a$ for all $t\in\vI$. Let $U=f^{-1}([0,\varepsilon))$.

	For $s\in [0,1]$, $\alpha\in \partial\vN$ and $t\in[0,n]$ let
	\[
		G(s,\alpha)(t) = (1-s)\alpha(t) + s (g(f(\alpha),\alpha^1(t)), g(f(\alpha),\alpha^2(t)), \dots, g(f(\alpha),\alpha^n(t))).
	\]
	If $\alpha^j(t)=\varepsilon$ for $\varepsilon\in\{0,1\}$, then $g(f(\alpha),\alpha^j(t))=\varepsilon$, which implies that $G(s,\alpha)(t)\in \vI^{{\vee}\bn}$. Thus,   this formula 	defines the continuous map $G:[0,1] \times U \to \vP$. Now the composition $\nat \circ G$ is a strong deformation retraction of $U$ into $\partial\vN$. Indeed, for every $\alpha\in \partial\vN$ we have $f(\alpha)=0$ and, therefore, $\nat(G(\alpha))=\nat(\alpha)=\alpha$. If $\alpha\in U$, then there exists $k\in\Free(\bn)$ such that
	\[
		\min\{\alpha^j(k), 1-\alpha^j(k)\} \leq f(\alpha)
	\]
	for all $j\in\{1,\dots,n\}$, i.e., $\alpha^j(k)\in [0,f(\alpha)]\cup [1-f(\alpha),1]$. Thus,
	\[
		G(1,\alpha)(k)=(g(f(\alpha), \alpha^1(k)), \dots, g(f(\alpha), \alpha^n(k)))
	\]
	is a vertex. Furthermore,
	\begin{multline*}
		\left| k - \len(G(1,\alpha)|_{[0,k]}) \right|=
		 \left|k- \sum_{j=1}^n g(f(\alpha), \alpha^j(k)) \right| =
		 \left|\sum_{j=1}^n \alpha^j(k)- \sum_{j=1}^n g(f(\alpha), \alpha^j(k)) \right| \\
		 \leq \sum_{j=1}^n \left| \alpha^j(k) - g(f(\alpha), \alpha^j(k))\right|
		 \leq nf(\alpha)
		  <1.
	\end{multline*}
	Thus, $\len(G(1,\alpha)|_{[0,k]})=k$ and hence $\nat(G(\alpha,1))(k) = G(\alpha,1)(k)$ is a vertex, which shows that $\nat\circ G(1,\alpha)\in \partial\vN$.
\end{proof}

\begin{prp}\label{p:QIsHtpEq}
	The map 
	\[
		Q_n^K:\Lhocolim_{\bc\in \Ch(K;n)}{\vN(\square^{{\vee}\bn^\bc})_{\bO}^\bI}
		\to
	{\Lcolim_{\bc\in \Ch(K;n)}  {\vN(\square^{{\vee}\bn^\bc})_{\bO}^\bI}}
	\]
	 is a homotopy weak equivalence.
\end{prp}
\begin{proof}
	By Propositions \ref{p:BoundaryCover} and \ref{p:PartialNCofibration}, the assumptions of the Proposition \ref{p:Dugger} are satisfied.
\end{proof}

\section{Comparison of colimit and the space of natural tame d-paths}

Fix a bi-pointed $\square$--set $K$ and an integer $n\geq 0$. The main goal of this section is to prove that the right-hand map in the sequence (\ref{e:MainSequence}) is a homotopy equivalence. Recall that $\Tam^K_n$ is the tamification map defined in (\ref{e:Tam}).

\begin{prp}\label{p:RegularTamePathCriterion}
	Let $\alpha\in \vN^t_{[0,n]}(K)_\bO^\bI$. Assume that $\Vertic(\alpha)=\Vertic(\bTam^K_n(\alpha))$. Then $\alpha$ is a regular natural tame d-path.
\end{prp}
\begin{proof}
	Let $\alpha=[\bc;\beta]$ be a minimal nt-presentation. Then also 
	\[
		\Tam^K_n(\alpha)=[\bc; \Tam^{\square^{{\vee}\bn^\bc}}_n(\beta)]
	\]
	 is a minimal nt-presentation, since $\Tam_n$ is functorial, and 
\[
	\Vertic(\Tam^K_n(\alpha))=\Vertic(\alpha)=\Vertic(\bn^\bc).
\]	 
	 Assume that $\alpha$ is not regular. Thus, there exists $i\in\{1,\dots,l(\bc)\}$ such that $\beta_i\in \vN_{[t^\bc_{i-1},t^\bc_i]}(\partial\vI^{n^\bc_i})_\bO^\bI$. Then 
\[
	 \Tam^{\square^{{\vee}\bn^\bc}}_n(\beta)_i:=\Tam^{\square^{{\vee}\bn^\bc}}_n(\beta)|_{[t^\bc_{i-1},t^\bc_i]}\in \vN_{[t^\bc_{i-1},t^\bc_i]}(\partial\vI^{n^\bc_i})_\bO^\bI
\] 
is tame (regarded as a d-path in $\partial\vI^{n^\bc_i}$) and, therefore, $k\in\Vertic(\bTam(\beta_i))$ for some $t^\bc_{i-1}<k<t^\bc_i$. This contradicts the assumption, since $k\not\in\Vertic(\alpha)$.
\end{proof}

\begin{prp}\label{p:QPreservesEquivalence}
	Let $\alpha\in \vN^t_{[0,n]}(K)_\bO^\bI$. If $\alpha=[\bc;\beta]$ and $\alpha=[\bc';\beta']$ are equivalent nt-presentations, then also $\bTam(\alpha)=[\bc;\bTam(\beta)]$ and $\bTam(\alpha)=[\bc';\bTam(\beta')]$ are equivalent nt-presentations.
\end{prp}
\begin{proof}
	This follows from the functoriality of $\Tam_n$.
\end{proof}

\begin{prp}\label{p:GluingLemma}
	Let $\alpha\in\vN^t_{[0,n]}(K)_\bO^\bI$ and let $\alpha=[\bc;\beta]$ and $\alpha=[\bc';\beta']$ be nt-presentations of $\alpha$. Then the nt-presentations $\bTam^n(\alpha)=[\bc; \bTam^n(\beta)]$ and $\bTam^n(\alpha)=[\bc'; \bTam^n(\beta')]$ are equivalent.
\end{prp}
\begin{proof}
	We have an ascending sequence of sets
	\[
		\{0,n\}\subseteq \Vertic(\alpha)\subseteq \Vertic(\bTam(\alpha))\subseteq \Vertic(\bTam^2(\alpha))\subseteq \dots \subseteq \Vertic(\bTam^n(\alpha))\subseteq \{0,\dots,n\}.
	\]
	Hence, there exists $r\in\{0,\dots,n-1\}$ such that $\Vertic(\bTam^r(\alpha))=\Vertic(\bTam^{r+1}(\alpha))$. By Proposition \ref{p:RegularTamePathCriterion},  $\bTam^r(\alpha)$ is a regular tame path and hence, by Proposition \ref{p:RegularNTAreEquivalent}, the nt-presentations $\bTam^r(\alpha)=[\bc;\bTam^r(\beta)]$ and $\bTam^r(\alpha)=[\bc';\bTam^r(\beta')]$ are equivalent. Now the conclusion follows from Proposition \ref{p:QPreservesEquivalence}.
\end{proof}

\begin{prp}\label{p:GluingDiagram}
	Consider the diagram of solid arrows
	\begin{equation}
	\begin{diagram}
		\node{\Lcolim_{\bc\in \Ch(K;n)} \vN_{[0,n]}(\vI^{{\vee}\bn^{\bc}})_\bO^\bI}
			\arrow{e,t}{\Tam_n^{\square^{{\vee}\bn^\bc}}}
			\arrow{s,l}{F^K_n}
		\node{\Lcolim_{\bc\in \Ch(K;n)} \vN_{[0,n]}(\vI^{{\vee}\bn^{\bc}})_\bO^\bI}
			\arrow{s,r}{F^K_n}
	\\
		\node{\vN^t_{[0,n]}(K)_\bO^\bI}
			\arrow{e,t}{\Tam_n^K}
			\arrow{ne,t,..}{f}
		\node{\vN^t_{[0,n]}(K)_\bO^\bI}
	\end{diagram}
	\end{equation}
	There exists a map $f$ that makes the diagram commutative.
\end{prp}
\begin{proof}
	Propositions \ref{p:GluingLemma} and \ref{p:vNIsCompact} guarantee that the assumptions of Proposition \ref{p:QuotientMap} are satisfied. Thus, there exists a map $f$ such that the upper triangle commutes. Since the map $F_n^K$ is surjective, this implies that also the lower trangle commutes.
\end{proof}

\begin{prp}\label{p:ColimNHEq}
	The map
	\begin{equation*}
		F_n^K:\Lcolim_{\bc\in \Ch_n(K)_v^w} \vN_{[0,n]}(\vI^{{\vee}\bn^\bc})_\bO^\bI \to  \vN^t_{[0,n]}(K)_v^w
	\end{equation*}
is a homotopy equivalence.
\end{prp}
\begin{proof}
	This follows immediately from Proposition \ref{p:GluingDiagram}, since both horizontal maps are homotopic to the identities on the respective spaces.
\end{proof}

\section{Topological lemmas}

\begin{prp}\label{p:LipschitzMappingSpaceIsCompact}
	Let $X,Y$ be metric spaces and let $L>0$. Let $\Lip_L(X,Y)$ denote the space of $L$--Lipschitz maps $X\to Y$ with compact-open topology and let $Y^X$ be the space of all continuous maps $X\to Y$ with the product topology. Assume that $Y$ is compact. Then the inclusion map $\Lip_L(X,Y)\xrightarrow\subseteq Y^X$ is a homeomorphism on its image.
\end{prp}
\begin{proof}
	We need to prove the following statement: for every compact subset $K\subseteq X$, an open subset $U\subseteq Y$ and $f\in\Lip_L(X,Y)$ such that $f(K)\subseteq U$ there exist sequences $x_1,\dots,x_n$ of points of $X$ and $U_1,\dots,U_n$ of open subsets of $Y$ such that
	\[
		\forall_{g\in \Lip_L(X,Y)} \; \big( (\forall_{i=1}^n\; g(x_i)\in U_i) \; \Rightarrow\; g(K)\subseteq U\big).
	\]
	Since $f(K)$ is compact, the distance $d$ between $f(K)$ and $Y\setminus U$ is strictly positive. Let $\{x_1,\dots,x_n\}$ be a family of points of $K$ such that the family of open balls having radius $d/2L$ and centered at $x_i$'s cover $K$. Let $U_i=B(f(x_i),d/2)$. Assume that $g\in  \Lip_L(X,Y)$ is a function such that $g(x_i)\in U_i$. For every $x\in K$ there exists $i\in\{1,\dots,n\}$ such that $\dist(x,x_i)<d/2L$. Therefore,
	\[
		\dist_Y(g(x), f(x_i))\leq \dist_Y(g(x), g(x_i)) + \dist_Y(g(x_i), f(x_i)) \leq L\dist_X(x,x_i)+d/2\leq d.
	\]
	As a consequence, $g(x)\in U$, which ends the proof.
\end{proof}

\begin{prp}\label{p:vNIsCompact}
	For every $\bn\in\Seq(n)$, the space $\vN_{[0,n]}(\vI^{{\vee}\bn})_\bO^\bI$ is compact.
\end{prp}
\begin{proof}
	Since
	\[
		\vN_{[0,n]}(\vI^{{\vee}\bn})_\bO^\bI\cong \prod_{i=1}^{l(\bn)} \vN_{[t^\bn_{i-1}, t^{\bn}_i]}(\vI^{n_i})_\bO^\bI,
	\]
	it is sufficient to prove that $\vN_{[0,n]}(\vI^n)_\bO^\bI$ is compact. But $\vN_{[0,n]}(\vI^n)_\bO^\bI$ is a closed subset of $\Lip_1(I,I^n)$, with $L^1$--metric on $I^n$, so the conclusion follows from Proposition \ref{p:LipschitzMappingSpaceIsCompact}.
\end{proof}

\begin{prp}\label{p:QuotientMap}
	Consider the diagram of topological spaces and continuous maps
	\[
		\begin{diagram}
			\node{X}
				\arrow{e,t}{g}
				\arrow{s,l}{p}
			\node{Z}
		\\
			\node{Y}
				\arrow{ne,b,..}{f}
		\end{diagram}
	\]
	Assume that
	\begin{itemize}
		\item{$X$ is compact.}
		\item{$p$ is surjective.}
		\item{For $x,x'\in X$, $p(x)=p(x')$ implies $g(x)=g(x')$.}
	\end{itemize}
	Then there exists a unique continuous map $f$ making this diagram commutative.
\end{prp}
\begin{proof}
	The existence and the uniqueness of $f$ are straightforward. Any surjective map from a compact space is open; thus, for every open subset $U\subseteq Z$, the subset $f^{-1}(U)= p(g^{-1}(U))\subseteq Y$ is open. Hence, $f$ is continuous.
\end{proof}

\end{document}